\documentclass[12pt,a4paper]{article}

\usepackage[dvips]{graphicx}
\usepackage{amssymb,latexsym}
\usepackage{mathrsfs}
\usepackage{psfrag}
\usepackage[all]{xy}
\usepackage{color}
\usepackage{verbatim}
\usepackage{graphicx,geometry,color,epsfig,epstopdf}
\usepackage{caption}
\usepackage{subcaption}

\newtheorem{prop}{Proposition}[section]
\newtheorem{theo}[prop]{Theorem}

\newtheorem{rem}[prop]{Remark}

\newtheorem{lem}[prop]{Lemma}

 \newcommand{\bx}{\mbox{\boldmath $x$}}

\newcommand{\bv}{\mbox{\boldmath $v$}}

\newcommand{\bu}{\mbox{\boldmath $u$}}

\newcommand{\bN}{\mbox{\boldmath $N$}}

\newcommand{\R}{\mbox{\boldmath $\mathbb{R}$}}

\newenvironment{proof}
{\begin{trivlist} \item[\hskip \labelsep {\bf Proof}\hspace*{3 mm}]}
	{\hfill$\Box$\end{trivlist}}
\newenvironment{acknow}
{\begin{trivlist} \item[\hskip \labelsep {\bf Acknowledgments.}]}
	{\end{trivlist}}

\date{}

\begin{document}

\title{Bifurcations of robust features on surfaces in the Minkowski 3-space}
\author{Marco Ant\^onio do Couto Fernandes}

\maketitle
\begin{abstract}
We obtain the bifurcation of some special curves on generic 1-parameter families of surfaces in the Minkowski 3-space.  
The curves treated here are the locus of points where the induced pseudo metric is degenerate, the  discriminant of the lines principal curvature, 
the parabolic curve and the locus of points where the mean curvature vanishes. 
\end{abstract}

\renewcommand{\thefootnote}{\fnsymbol{footnote}}

\section{Introduction}\label{sec:intro}


We consider generic 1-parameter families of surfaces in the Minkowski 3-space $\mathbb R^3_1$. The Lorentzian metric in $\mathbb R^3_1$ induces a pseudo metric on a surface in $ M\subset \mathbb R^3_1$. The locus of points on $M$ where the induced pseudo metric is degenerate is called the {\it Locus of Degeneracy} and is denoted by the $LD$ for short. 
The $LD$ separates the surface into regions where the induced metric is Riemannian or Lorentzian. 
On such regions the Gaussian and mean curvature of the surface are well defined. The zero sets of the Gaussian and mean curvature functions extends to the $LD$ and are called, respectively,  the {\it parabolic curve} and {\it mean curvature null curve} ($PC$ and $MCNC$ for short).

The equation of the lines of principal curvature also extends to the $LD$ (\cite{IzumiyaTari}). In the Riemannian region the discriminant of that equation consists of the umbilic points. In the Lorentzian region it is the locus of points where two principal directions coincide and become lightlike and is labelled {\it Lightlike Principal Locus}, $LPL$ for short.

We obtain in this paper the catalogue of all possible local bifurcations of the $LD$, the $LPL$, the $MCNC$ and the $PC$ when the surface is deformed in generic 1-parameter families of surfaces. 
We deal in \S \ref{sec:loretz} with the bifurcations at points  in the Lorentzian region and in \S \ref{sec:ld} at points on the $LD$. We give some preliminaries in \S \ref{sec:prel} and define the notion of genericity and codimension in \S \ref{sec:genericity&codim}. We observe that at points on the Riemannian region of the surface the situation is identical to that of surfaces in the Euclidean space. Families of such surfaces are studied in, for example,  \cite{BruceGiblinTari0,BruceGiblinTari,BruceGiblinTari2,BruceGiblinTari3,jorge,GarciaSoto,GarciaSotoGut}.
\section{Preliminaries} \label{sec:prel}

The {\it Minkowski space} $(\mathbb{R}_1^{3},\langle ,\rangle )$ is the vector
space $\mathbb{R}^{3}$ endowed with the metric induced by the pseudo-scalar product
$
\langle \bu, \bv\rangle =u_0v_0+u_1v_1-u_2v_2$, for any vectors $\bu=(u_0,u_1,u_2)$ and $\bv =(v_0,v_1, v_2)$ in
$\mathbb{R}^{3}$ (see for example \cite{Couto_Lymb, oNeill} for a treatment of the geometry of surfaces in $\mathbb R^3_1$).
A non-zero vector $\bu\in \mathbb R^3_1$ is said to be {\it spacelike} if
$\langle \bu, \bu\rangle>0$, {\it lightlike} if
$\langle \bu, \bu\rangle=0$ and {\it timelike} if $\langle \bu, \bu\rangle<0$.
The norm of a vector $\bu\in \mathbb{R}_1^{3}$ is defined by
$\Vert \bu \Vert=\sqrt{ |\langle \bu, \bu\rangle|}.$

Let $M$ be a smooth and regular surface in $\mathbb R^3_1$ and let 
$\bx:U\subset \mathbb R^2\to \mathbb{R}_1^{3}$ be a local parametrisation of $M$.
We shall simplify notation  and write $\bx(U)=M$. Let
$$E=\langle {\bx}_u,{\bx}_u \rangle,\quad
F=\langle{\bx}_u,{\bx}_v\rangle,\quad
G=\langle {\bx}_v,{\bx}_v\rangle$$
denote the coefficients of the first fundamental form of $M$ with respect to the parametrisation $\bx$, where subscripts denote partial derivatives. The induced (pseudo) metric on $M$ is Lorentzian (resp. Riemannian, degenerate) at ${\rm p}=\bx(u,v)$ if and only if $\delta(u,v) = (F^2-EG)(u,v)>0$ 
(resp. $<0$, $=0$). The locus of points on the surface where the metric is degenerate 
is called the {\it locus of degeneracy} and is denoted by $LD$.  
We identify the $LD$ on $M$ with its pre-image in $U$ by $\bx$. Then the $LD$ (in $U$) is given by
$$
LD=\{(u,v)\in U\, |\, \delta(u,v)=0\}.
$$

A direction  $(du,dv)\in T_{\rm p}M$ is lightlike if and only if
\begin{equation}\label{eq:LightBDE}
	Edu^2+2Fdudv+Gdv^2=0.
\end{equation}
Equation (\ref{eq:LightBDE}) has two  (resp. no) solutions when ${{\rm p}}$ is in the Lorentzian (resp. Riemannian) region of $M$.
At points on the $LD$, there is a unique (double) solution of the equation.

\begin{theo}{\rm (\cite{IzumiyaTari})}\label{the:param_lightlike}
There is a local parametrization $\bx : U \subset \R^2 \to \R^3_1$ of $M$ at $p \in LD$ such that the lightlike directions in $T_qM$ is given by $\mathbb{R}.x_u$ for every $q \in LD$, i.e., $E = F = 0$ on the $LD$.

If $q \in M$ belongs to the Lorentzian region, there is a local parametrization of $M$ at $p$ such that $\mathbb{R}.x_u$ and $\mathbb{R}.x_v$ are the lightlike directions in $T_qM$ for every $q \in \bx(U)$, i.e., $E = G = 0$ on $U$.
\end{theo}

At ${{\rm p}} \in M\setminus LD$, we have a well defined unit normal vector (the Gauss map) $\bN=\bx_u\times\bx_v/||\bx_u\times\bx_v||$,
which is timelike (resp. spacelike) if ${\rm p}$ is in the Riemannian (resp. Lorentzian) region of $M$.
(See \cite{pei} for a definition of an $\mathbb RP^2$-valued Gauss map.) The map $A_{{\rm p}}=-d\bN_{{\rm p}}:T_{{\rm p}}M\to T_{{\rm p}}M$ is a self-adjoint
operator on $M\setminus LD$. We denote by
$$l=\langle \bN,\bx_{uu}\rangle,\quad m=\langle \bN,\bx_{uv}\rangle, \quad n=\langle \bN,\bx_{vv}\rangle$$
the coefficients of the second fundamental form on $M\setminus LD$.

When the eigenvalues $\kappa_1$ and $\kappa_2$ of $A_p$ are real, they are called the {\it principal curvatures} and their associated eigenvectors the {\it principal directions}
of $M$ at ${{\rm p}}$. There are always two principal curvatures at each point on the Riemannian part of $M$ but this is not always
true on its Lorentzian part.
A point ${\rm p}$ on $M$ is called an {\it umbilic point} if  $\kappa_1=\kappa_2$ at ${\rm p}$ (i.e., if $A_p$ is a multiple of the identity map). It is called a {\it spacelike umbilic point} (resp. {\it timelike umbilic point}) if ${\rm p}$ is in the Riemannian (resp. Lorentzian) 
part of $M$.

The \textit{lines of principal curvature}, which are the integral curves of the principal directions, are the solutions of the binary differential equation (BDE)
\begin{equation}\label{eq:principalBDE}
	(Fn-Gm)dv^2+(En-Gl)dvdu+(Em-Fl)du^2=0.
\end{equation}
The discriminant of BDE (\ref{eq:principalBDE}) 
$$\{
(u,v)\in U \, | \left((En-Gl)^2-4(Fn-Gm)(Em-Fl)\right)(u,v)=0\}
$$
consists of the umbilic points in the Riemannian region of $M$ and is the locus of points in the Lorentzian region where two principal directions coincide and become lightlike. It is labelled {\it Lightlike Principal Locus} $(LPL)$ in \cite{Izu-Machan-Farid}. The principal directions are orthogonal when there are two of them at a given point; in particular one is spacelike and the other is timelike if the point is in the Lorentzian region.

The \textit{Gaussian curvature} and \textit{mean curvature} of $M$ at $p$ are $K = \det (A_p)$ and $H=-\frac{1}{2} tr (A_p)$, respectively. We have
$$K = \frac{ln-m^2}{EG-F^2}, \qquad H = \frac{lG-2mF+nE}{2 (EG-F^2)}.$$
The locus of points where the Gaussian and mean curvature are zero are called the \textit{parabolic curve} ($PC$) and \textit{mean curvature null curve} ($MCNC$), respectively.

One can extend the lines of principal curvature, the $LPL$, the $PC$ and the $MCNC$ across the $LD$ as follows (\cite{IzumiyaTari}).
As equation (\ref{eq:principalBDE}) is homogeneous in $l,m,n$,
we can multiply these coefficients by $||\bx_u\times\bx_v||$ and substitute them in the equation by
\[
\bar{l}=\langle \bx_u\times\bx_v,\bx_{uu}\rangle, \quad
\bar{m}=\langle \bx_u\times\bx_v,\bx_{uv}\rangle,\quad
\bar{n}=\langle \bx_u\times\bx_v,\bx_{vv}\rangle.
\]
The new equation 
\begin{equation}\label{eq:principalLD}
	(G\bar{m}-F\bar{n})dv^2+(G\bar{l}-E\bar{n})dudv+(F\bar l-E\bar{m})du^2=0
\end{equation}
is defined at points on the $LD$ and
its solutions are the same as those of equation (\ref{eq:principalBDE}) in the Riemannian and Lorentzian regions of $M$. A {\it lightlike umbilic point}  is defined as a point on the $LD$ where all the coefficients of equation  (\ref{eq:principalLD}) vanish (see \cite{CaratheodoryR31}). 
This occurs if and only if the $LD$ is singular (\cite{CaratheodoryR31}). The discriminant of the BDE $(\ref{eq:principalLD})$ extends the $LPL$ to points on the $LD$. Thus, a point $p = \bx(u,v)$ belongs to the $LPL$ of $M$ when
$$
\tilde \delta(u,v) = \left((E\bar n-G\bar l)^2-4(F\bar n-G\bar m)(E\bar m-F\bar l)\right)(u,v)=0.
$$

For the $PC$, note that $K = 0$ if and only if $\bar l \bar n-\bar{m}^2=0$. Thus, the $PC$  extends to  points on the $LD$ as  the zero set of the function $\bar K=\bar l \bar n-\bar{m}^2$. Similarly, the $MCNC$  extends to the $LD$ as the zero set of the function $\bar H = \bar l G-2 \bar m F+\bar n E$.

\medskip

We use here some concepts from singularity theory (see \cite{arnold,wall}). Two germs of functions $f_i:(\mathbb R^n,0)\to (\mathbb R^m,0), i=1,2$ are said to be \textit{$\mathcal R$-equivalent} if $f_2=f_1\circ h^{-1}$ for some germ of a diffeomorphism $h : (\R^n,0) \to (\R^n,0)$. Similarly, $f_1$ and $f_2$ are said to be $\mathcal{L}$-\textit{equivalent} when there is a germ of a diffeomorphism $k:(\R^m,0) \to (\R^m,0 )$ such that $f_2=k \circ f_1$. Finally, $f_1$ and $f_2$ are $\mathcal{A}$-\textit{equivalent} if there exist germs of diffeomorphisms $h$ and $k$ such that $f_2=k \circ f_1 \circ h^{- 1}$.

The  simple singularities of germs of functions ($m=1$) are classified by Arnold \cite{arnold}. Representatives of their $\mathcal R$-orbits when $n=2$ are as follows 
\[ 
A_k: \pm(x^2\pm y^{k+1}),k\ge 1,\, D_k:x^2y\pm y^{k-1},k\ge 4,\, E_6:x^3+y^4,\, E_7:x^3+xy^3,\, E_8:x^3+y^5.
\]
For $n\ge 3$ one adds the quadratic form $\pm x_3^2\pm \cdots\pm x_n^2$ to the above normal forms.

We denote by $j^k_qf$ the \textit{$k$-jet} of a germ of a smooth/analytic function $f:(\mathbb R^n,q)\to \mathbb R$, i.e., its Taylor polynomial of order $k$ at $q$, and by $J^k(n,1)$ the space of all $k$-jets of germs of functions. If $q$ is the origin, then we denote the $k$-jet by $j^kf$. When $n=2$, we write the $k$-jet, at the origin, of a function $f$ in the form $j^kf=\sum_{s=0}^k\sum_{i=0}^s a_{si}x^{s-i}y^i$
and identify it with the coefficients $(a_{si}) \in \R^{(k+1)(k+2)/2}$.


\section{Genericity and codimension}\label{sec:genericity&codim}

In this section, we present some concepts used in the rest of the paper. Let $G(3,1)$ be the group generated by Lorentzian rotations, reflection, translations and homotheties with the composition operation. 
Any $\sigma \in G(3,1)$ preserves the $LD$, $LPL$, $PC$ and $MCNC$ on surfaces  in $\mathbb R^3_1$.


Let $p \in M$ be a point in Lorentzian region. 
After composing by elements of $G (3,1)$ if necessary, 
we can parametrise $M$ locally in Monge-form 
$\bx: U \subset \R^2 \to \R^3_1$, with $\bx(x,z) = (x,f(x,z),z)$, 
so that $p = \bx(0,0)$ and
\begin{equation}\label{eq:param_monge_time}
	j^kf = \sum_{i=2}^k \sum_{j=0}^i a_{ij} x^{i-j}z^j.
\end{equation}
Then, $p$ is an umbilic point if and only if $a_{20} = - a_{22}$ and $a_{21} = 0$.

Similarly, if $p \in M$ belongs to $LD$, then after composing by elements of $G (3,1)$ if necessary, we can parametrise $M$ locally in Monge-form $\bx: U \subset \R^2 \to \R^3_1$ with  $\bx(x,y) = (x,y,f(x,y))$, so that $p = \bx(0,0)$ and
\begin{equation}\label{eq:param_monge_light}
	j^kf = x + \sum_{i=2}^k \sum_{j=0}^i a_{ij} x^{i-j}y^j.
\end{equation}
With the above setting, $p \in LPL$ if and only if $a_{20} = 0$, and $p$ is an umbilic point if and only if $a_{20} = a_{21} = 0$.

%

In \cite{bruce84}, Bruce described a technique for studying local properties of surfaces 
in the Euclidean space $\mathbb R^3$ called Monge-Taylor map. 
We follow the ideas in \cite{bruce84} but, as our study is local, 
we choose a fixed coordinate system  $\{ {\bf u}, {\bf v}, {\bf w}\}$ in a neighbourhood $W$ of a point ${\rm p}\in M$ with ${\bf w}$ transverse to $M$ at points in $W$. 
For each ${\rm p} \in W$, consider the coordinate system $\{ {\bf u}, {\bf v}, {\bf w}\}$ with origin at ${\rm p}$ and take a local parametrization of $ M$ in Monge form $(x,y)\mapsto \bx_{{\rm p}}(x,y)=(x,y,f_{{\rm p}}(x,y))$. 
We define the \textit{Monge-Taylor map} 
$\theta: W\to J^k(2,1)$ by $\theta({{\rm p}})=j^kf_{{\rm p}}$.



\subsection{Generic properties of surfaces}\label{sec:generic}


Let $\mathcal{G}$ be the set of germs $(M,p)$ where $M \subset \R^3_1$ is a smooth surface and $p \in M$. A germ $\bx : (\R^2,0) \to (\R^3_1,p)$ is a \textit{local parametrization} if $\bx$ is smooth, has a continuous inverse and $d\bx_{(0,0)} : \R^2 \to \R^3$ is one-to-one. A \textit{local parametrization of $M$} is a local parametrization whose image is contained in $M$. Consider
$$\mathcal{F} = \{ \bx : (\R^2,0) \to (\R^3_1,p) : p \in \R^3_1, \, \bx \textit{ is a local parameterization} \}$$
with the Whitney topology. Note that $\sim_\mathcal{R}$ is an equivalence relation on $\mathcal{F}$ and denote by $[\bx]$ the equivalence class of $\bx \in \mathcal{F}$ relative to $\sim_{\mathcal{R}}$. Therefore, $\mathcal{F}/ \sim_\mathcal{R}$ is a topological space with the quotient topology.

%
%
%

The map $f: \mathcal{G} \to \mathcal{F} / \sim_\mathcal{R}$ that associates with $(M,p) \in \mathcal{G}$ the equivalence class $[ \bx: (\R^2,0) \to (\R^3_1,p)]$ of a local parameterization of $M$ is well defined and is a bijection. Therefore, we provide $\mathcal{G}$ with the only topology that makes $f$ a homeomorphism. A local property $P$ is said to be \textit{generic} in $\mathcal{G}$ when the subset $\mathcal{P} \subset \mathcal{G}$ formed by the pairs $(M,p)$ such that the point $p$ of $M$ satisfies the property $P$ is open and dense in $\mathcal{G}$. Given a subset $\tilde \mathcal{G} \subset \mathcal{G}$ we consider the induced topology of $\mathcal{G}$ in $\tilde \mathcal{G}$ and define the concept of generic property in $\tilde \mathcal{G}$ as in $\mathcal{G}$. For example, if
$$\tilde \mathcal{G} = \{ (M,p) \in \mathcal{G} : p \textit{ is a timelike umbilic point}\},$$
then the property $P : ``LPL \textit{ has a singularity } A_1^- \textit{ at } p"$ is generic in $\tilde \mathcal{G}$ \cite{MarcoTari}.

A $m$-parameters family of surfaces $M_t$ deforming $(M,p) \in \mathcal{G}$ is a family of surfaces whose parameterizations $\bx_t : U \to \R^3_1$ of $M_t $ smoothly depend of $t = (t_1,\dots,t_m) \in (\R^m,0)$ and $M_0 = M$. Let $z=f(x,y)$ be a Monge parameterization of the surface $M$. Then any $m$-parameter family of surfaces $M_t$ with $M_0 = M$ can be given by
$$z=f(x,y)+h(x,y,t_1,\dots,t_m)$$
with $h$ differentiable and $h(x,y,0,\dots,0) = 0$ for all $(x,y)\in U$. We say that the family $M_t$ is generic when $h$ satisfies generic conditions (open and dense conditions).

\begin{rem}
	When the surface $M$ is given in Monge form $x=f(y,z)$ or $y=f(x,z)$, the definition of a generic family of surfaces is analogous.
\end{rem}

We denote by $LD_t$ the set of points on the surface $M_t$ where the metric is degenerate. Thus, $LD_0$ is the $LD$ of the surface $M$ and the family of curves $LD_t$, with $t \in (\R^m,0)$, is a deformation of the $LD$ of $M$. Similarly, the $LPL$, the $PC$ and the $MCNC$ of $M$ are deformed along the family of surfaces. We denote the $LPL$, the $PC$ and the $MCNC$ of $M_t$ by $LPL_t$, $PC_t$ and $MCNC_t$, respectively.

\subsection{Codimension of a property}\label{sec:codimension}

Let $M \subset \R^3_1$ be a surface. A point at $M$ with a property $(P)$ has \textit{codimension $n$} when the submanifold in $J^k(2,1)$, for $k$ sufficiently large, obtained by the conditions imposed by $(P)$ on the coefficients of the Monge parameterization has codimension $n+2$. Generically, a $n$-parameter family of surfaces has points with properties of codimension $\leq n$.
%
We are interested in the bifurcations of the $LD$, the $LPL$, the $PC$ and the $MCNC$ in generic 1-parameter families of surfaces, that is, in properties on these curves with codimension 0 and 1.

Properties of codimension 0 with respect to these curves are stable and known. For example, regular points of the $LD$ have codimension 0. In fact, taking $M$ in Monge form $z=f(x,y)$, with
$j^2f=a_{00}+a_{10}x+a_{11}y+a_{20}x^2+a_{21}xy+a_{22}y^2$. Then $p = \bx(0,0)$ is a regular point of the $LD$ when
\begin{equation}\label{eq:ld_reg}
	a_{10}^2+a_{11}^2=1,\quad a_{11} a_{21} + 2 a_{10} a_{20} = 0,\quad	2 a_{11} a_{22} + a_{10} a_{21} \neq 0.
\end{equation}
As the submanifold of $J^2(2,1)$ defined by (\ref{eq:ld_reg}) has codimension 2, regular points of the $LD$ have codimension 0. Similarly, regular points of the $LPL$, $PC$ and $MCNC$ have codimension 0. Also, the Morse singularities of the $LPL$ have codimension 0 (\cite{IzumiyaTari,sotogut}). 

Given two curves $\gamma_1$ and $\gamma_2$ on $M$, we denote by $m(\gamma_{1},\gamma_{2}:p)$ the order of contact between then at $p$. 
At points of codimension 0, the $LPL$ has ordinary tangencies with the $LD$ and with the $PC$ \cite{IzumiyaTari}, the $LD$ is transversal to the $PC$ (see \S \ref{sec:int_ld_K}) and the $MCNC$ is transversal to all other curves (see \S \ref{sec:int_lpl_k} and \S \ref{sec:int_ld_lpl}).

We study here properties of codimension 1. We separete the cases when the point is considered in Lorentzian region (\S \ref{sec:loretz})  or on the $LD$ (\S \ref{sec:ld}). The case when the point is in the Riemannian region are already studied elsewhere. At points in the Lorentzian region, codimension 1 cases can occur as follows:
\begin{itemize}
	\item[(i)] Degenerate contact between regular relevant curves (\S \ref{sec:int_lpl_k});
	\item[(ii)] flat timelike umbilic points (\S \ref{sec:flat_umb});
	\item[(iii)] singularities of the relevant curves (\S \ref{sec:umb_A3}, \S \ref{sec:sing_pc_mcnc}).
\end{itemize}
At points on the $LD$, we get codimension 1 cases when
\begin{itemize}
	\item[(i)] Degenerate contact between regular relevant curves with $LD$ (\S \ref{sec:int_ld_lpl}, \ref{sec:int_ld_K});
	\item[(ii)] singularities of the $LD$ (\S \ref{sec:umb_light}).
\end{itemize}

\section{Bifurcations at points in the Lorentzian region}\label{sec:loretz}

\subsection{Bifurcations at $\textbf{LPL} \cap \textbf{MCNC} \cap \textbf{PC}$}\label{sec:int_lpl_k}

In this subsection, we will discuss the intersection between the $LPL$ and the $PC$ when the $LPL$ is a regular curve (the singular cases are considered in \S \ref{sec:flat_umb} and \S \ref {sec:umb_light}). The intersection between the $LPL$ and the $PC$ at points of the $LD$ are always singularities of the $LPL$ (see \S \ref{sec:umb_light}). In the Riemannian region, the $LPL$ consist of the umbilic points and is singular at such points. Therefore, the intersection between the $PC$ and regular points of the $LPL$ occurs only in the Lorentzian region.

\begin{theo}\label{prop:inter_tripla}
	Let $M \subset \R^3_1$ be a smooth surface and $p \in M \setminus LD$. Then $p$ belongs to the intersection of two of three curves: the $LPL$, the $PC$ and the $MCNC$ if and only if $p$ belongs to the intersection of the three of then. Furthemore, if $p$ is a regular point on the $LPL$, then $m(LPL,PC:p) = 2 m(LPL,MCNC:p)$.
\end{theo}

\begin{proof}
	Let $\bx : U \to \R^3_1$ be the local parametrization of $M$ given by Theorem \ref{the:param_lightlike} with $p = \bx(0,0)$. Since $E = G = 0$ on $U$, then the $LPL$, the $PC$ and the $MCNC$ are given by $\bar l \bar n = 0$, $\bar m^2-\bar l \bar n = 0$ and $\bar m = 0$, respectively, and this proves the first part of the statement.
	
	Consider a local parametrization $\gamma : (\R,0) \to (\R^2,0)$ of the $LPL$ at a regular point $p = \bx(0,0)$. We have $\bar K (\gamma(x)) = \bar m (\gamma(x))^2$ and $\bar H (\gamma(x)) = -2 F(\gamma(x)) \bar m (\gamma(x))$. Since $F$ does not vanish on $U$, the order of contact between the $LPL$ and the $PC$ is twice the order of contact between the $LPL$ and the $MCNC$.
\end{proof}


Therefore, in this subsection, the intersections between the $LPL$ and the $MCNC$ will also be considered.

\begin{theo}\label{teo:lpl_k_codim}
	Regular points $p$ of the $LPL$ where $m(LPL,PC:p) > 2$ are of codimension $\geq 1$.
\end{theo}

\begin{proof}
	Take a local parameterization $\bx: U \to \R^3_1$ of $M$, with $\bx(x,z)=(x,f(x,z),z)$, $p = \bx(0,0)$ and
	$$j^3f = a_{00}+ a_{10} x+a_{11}z+a_{20} x^2+a_{21} x z+a_{22} z^2+a_{30} x^3+a_{31} x^2 z+a_{32} x z^2+a_{33} z^3.$$
	The point $p$ belongs to the $LPL \cap MCNC$ if and only if
	\begin{equation}\label{eq:lpl_mcnc}
		\left\{\begin{array}{lcc}
		a_{11}^2 a_{20}-a_{11} a_{10} a_{21}+a_{22} a_{10}^2+a_{22}-a_{20} & = & 0, \vspace{0.3cm}\\
		4 a_{11}^2 a_{20}^2-4 a_{11} a_{10} a_{21} a_{20}+a_{10}^2 a_{21}^2+a_{21}^2-4 a_{20}^2 & = & 0.
		\end{array}\right.
	\end{equation}
	These curves are transversal at $p$ when
	\begin{equation}\label{eq:lambda_1}
		\begin{array}{cclcc}
		\Lambda_1 & = & \left(a_{20}-a_{11}^2 a_{20}+a_{22} a_{10}^2+a_{22}\right) (a_{32} a_{31}-9 a_{33} a_{30})&& \vspace{0.2cm}\\
		&&+\left(a_{11}^2 a_{21}-2 a_{11} a_{22} a_{10}-a_{21}\right) \left(a_{31}^2-3 a_{32} a_{30}\right)&& \vspace{0.2cm}\\
		&&-\left(a_{10}^2 a_{21}-2 a_{11} a_{10} a_{20}+a_{21}\right) \left(a_{32}^2-3 a_{33} a_{31}\right)  & \neq & 0.
		\end{array}
	\end{equation}

	It follows from the Theorem \ref{prop:inter_tripla} that $m(LPL,PC:p)>2$ if $m(LPL,MCNC:p)>1$, that is, when $ \Lambda_1 = 0$. The result follows from the fact that the system of equations (\ref{eq:lpl_mcnc}), (\ref{eq:lambda_1}) define a submanifold on $J^3(2,1)$ of codimension 3.
\end{proof}


Let 
$$\mathcal{G}_1=\{(M,p) \in \mathcal{G}: m(LPL,PC:p) > 2\}.$$
Given $(M,p) \in \mathcal{G}_1$, take $\bx : U \subset \R ^2 \to \R^3_1$ the local parameterization of $M$ given by Theorem \ref{the:param_lightlike}, with $p = \bx(0,0)$. With this parameterization, we have $\tilde \delta = 4 F^2 \bar l \bar n$, $\bar K = \bar m^2-\bar l \bar n$ and $\bar H = -2 F \bar m$. We write
$$\begin{array}{ll}
j^kF = \sum_{s=0}^{k} \sum_{i=0}^{s} F_{si} u^{s-i} v^i, & j^k \bar l = \sum_{s=0}^{k} \sum_{i=0}^{s} l_{si} u^{s-i} v^i, \vspace{0.2cm}\\
j^k \bar m = \sum_{s=0}^{k} \sum_{i=0}^{s} m_{si} u^{s-i} v^i, & j^k \bar n = \sum_{s=0}^{k} \sum_{i=0}^{s} n_{si} u^{s-i} v^i.
\end{array}$$
We have $F_{00} \neq 0$ and $m_{00} = 0$ because $p \in MCNC$. On the other hand, $l_{00} = 0$ or $n_{00} = 0$ for $p \in LPL$, wit as $p$ is a regular point on the $LPL$. We assume without loss of generality that $l_{00} = 0$. Then $\tilde \delta_u(0,0) = 4 F_{00}^2 l_{10} n_{00}$ and $\tilde \delta_v(0,0) = 4 F_{00}^2 l_{11} n_{00}$, consequently $l_{10} \neq 0$ or $l_{11} \neq 0$. Suppose that $l_{10} \neq 0$, the other case is similar. We have $m_{11} = \frac{l_{11} m_{10}}{l_{10}}$ because the $LPL$ and the $MCNC$ are not transverse at $p$.

It follows from the Implicit Function Theorem that the $LPL$ (in $U$) given by $\tilde \delta(u,v) = 0$ is parameterized by $\gamma(v) = (g(v),v)$ in a neighborhood $V$ of the origin, for some smooth function $ g$. We get
$$j^4g=-\frac{l_{11}}{l_{10}} x-\frac{l_{11}^2 l_{20} - l_{10} l_{11} l_{21} + l_{10}^2 l_{22}}{l_{10}^3} x^2+\frac{c_3}{l_{10}^5} x^3+ \frac{c_4}{l_{10}^7} x^4,$$
with 
$$\begin{array}{ccl}
c_3 & = & -l_{10}^4 l_{33}+l_{10}^3 l_{11} l_{32}+l_{10}^3 l_{21} l_{22}-l_{10}^2 l_{11}^2 l_{31}-2 l_{10}^2 l_{11} l_{20} l_{22}-l_{10}^2 l_{11} l_{21}^2+l_{10} l_{11}^3 l_{30} \vspace{0.2cm}\\
&& +3 l_{10} l_{11}^2 l_{20} l_{21}-2 l_{11}^3 l_{20}^2, \vspace{0.3cm}\\
c_4 & = & -l_{10}^6 l_{44}+l_{10}^5 l_{11} l_{43}+l_{10}^5 l_{21} l_{33}+l_{10}^5 l_{22} l_{32}-l_{10}^4 l_{11}^2 l_{42}-2 l_{10}^4 l_{11} l_{20} l_{33}-2 l_{10}^4 l_{11} l_{21} l_{32}\vspace{0.2cm}\\
&& -2 l_{10}^4 l_{11} l_{22} l_{31}-l_{10}^4 l_{20} l_{22}^2-l_{10}^4 l_{21}^2 l_{22}+l_{10}^3 l_{11}^3 l_{41}+3 l_{10}^3 l_{11}^2 l_{20} l_{32}+3 l_{10}^3 l_{11}^2 l_{21} l_{31}\vspace{0.2cm}\\
&& +3 l_{10}^3 l_{11}^2 l_{22} l_{30}+6 l_{10}^3 l_{11} l_{20} l_{21} l_{22}+l_{10}^3 l_{11} l_{21}^3-l_{10}^2 l_{11}^4 l_{40}-4 l_{10}^2 l_{11}^3 l_{20} l_{31}-4 l_{10}^2 l_{11}^3 l_{21} l_{30}\vspace{0.2cm}\\
&& -6 l_{10}^2 l_{11}^2 l_{20}^2 l_{22}-6 l_{10}^2 l_{11}^2 l_{20} l_{21}^2+5 l_{10} l_{11}^4 l_{20} l_{30}+10 l_{10} l_{11}^3 l_{20}^2 l_{21}-5 l_{11}^4 l_{20}^3.\vspace{0.2cm}
\end{array}$$

%
%

\begin{theo}\label{teo:lpl_k_mcn}
	Let $(M,p) \in \mathcal{G}_1$ is generic. Then 
	\begin{description}
		\item[(1)] $m(LPL,PC:p) = 4$ and $m(LPL,MCNC:p) = 2$. There are two possible configurations for the three curves as shows in Figure \ref{fig:lpl_k_cmn_def} middle figures.
		\item[(2)] For a generic 1-parameter family of surfaces $M_t$, with $M_0=M$, the configurations of the three curves deform as in Figure \ref{fig:lpl_k_cmn_def}. 
	\end{description}
	\begin{figure}[h!]
		\centering
		\includegraphics[width=0.9\textwidth]{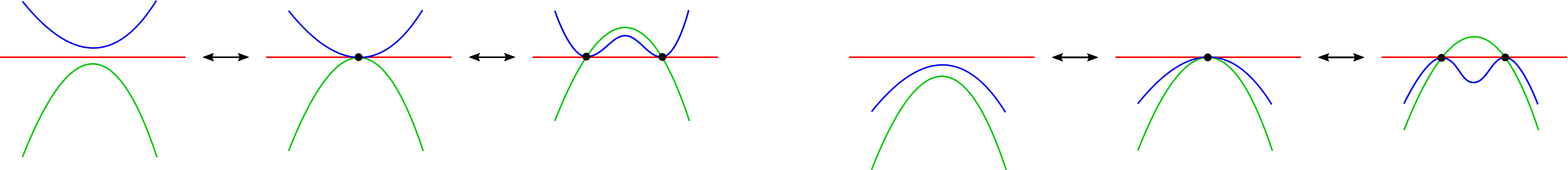}
		\caption{Deformations of the $LPL$ (red), the $PC$ (blue) and the $MCNC$ (green) in a generic 1-parameter family of surfaces at points $p$ where $m(LPL,MCNC) = 2$.}
		\label{fig:lpl_k_cmn_def}
	\end{figure}
\end{theo}

\begin{proof}
	\textbf{(1)} We observe that 
	$$\left(\bar K \circ \gamma\right) (v) = \left(\frac{\Lambda_2}{l_{10}^3} v^2\right)^2+O(5) \quad {\rm and} \quad \left(\bar H \circ \gamma\right) (v) = -2 F_{00} \frac{\Lambda_2}{l_{10}^3} v^2+O(3),$$	
	where $O(k)$ is a remamber of order $k$, and
	$$\Lambda_2 = l_{10}^3 m_{22}-l_{10}^2 l_{11} m_{21}-l_{10}^2 l_{22} m_{10}+l_{10} l_{11}^2 m_{20}+l_{10} l_{11} l_{21} m_{10}-l_{11}^2 l_{20} m_{10}.$$
	We have two cases depending on the sign of $F_{00} l_{10} \Lambda_2$. Figure \ref{fig:lpl_k_cmn_def} first when $F_{00} l_{10} \Lambda_2<0$ and last when $F_{00} l_{10} \Lambda_2>0$.

	\noindent \textbf{(2)} We take $\bx_t : U \to \R^3_1$ a parameterization of $M_t$ given by Theorem \ref{the:param_lightlike} which depends smoothly on $t \in (\R,0)$ and $\bx_0 = \bx$. Denote by $E(u,v,t)$, $F(u,v,t)$, $G(u,v,t)$, $\bar l(u,v,t)$, $\bar m(u,v,t)$ and $\bar n(u,v,t)$ the coefficients of the first and second fundamental forms of $M_t$ in $(u,v) \in U$ with respect to $ \bx_t$. Thus, $E(u,v,t) = G(u,v,t) = 0$ on $U$.
	
	The point $\bx_t(u,v)$ is on the $LPL_t$ if and only if $l(u,v,t)=0$. Likewise, $\bx_t(u,v) \in MCNC_t$ when $m(u,v,t) = 0$. Let $\sigma_t(u,v) = \left( l(u,v,t), m(u,v,t) \right)$. If $l(u,v,t) = \bar l(u,v) + t l_2(u,v,t)$ and $m(u,v,t) = \bar m(u,v) +t m_2(u,v,t)$, then
	$$\sigma_t \sim_\mathcal{A} \left( u,v^2+\frac{l_{10} m_2(0,0,0)-m_{10} l_2(0,0,0))}{l_{10}} t+O(2) \right),$$
	where $O(2)$ is a remander of order 2 in $t$. Therefore, when $l_{10} m_2(0,0,0)\neq m_{10} l_2(0,0,0)$,
	the $LPL_t$ and the $MCNC_t$ intersect at 2 or 0 points depending on the sign of $t$. As the $PC$ intersects these curves at $LPL_t \cap MCNC_t$, it follows that the deformation of these curves is as in Figure \ref{fig:lpl_k_cmn_def}.
\end{proof}

\subsection{Bifurcations at a flat timelike umbilic point}\label{sec:flat_umb}

A flat umbilic point is an umbilic point where Gaussian curvature is zero. Thus, flat umbilic points are where the $PC$ passes through a singularity of the $LPL$. Flat spacelike umbilic points are similar to flat umbilic points on surfaces in Euclidean space and such points have already been studied in \cite{BruceGiblinTari}.


%
%

Let 
$$\mathcal{G}_2 = \{(M,p) \in \mathcal{G} : p {\rm \,\, is \,\, a \,\, flat \,\, timelike \,\, umbilic \,\, point}\}.$$
In this subsection, unless otherwise noted, we use the local parameterization $\bx: U \to \R^3_1$ of $M$ given as in (\ref{eq:param_monge_time}) with $p = \bx(0,0)$. Since $p$ is a flat timelike umbilic point, it follows that $j^2f = 0$.

\begin{theo}\label{teo:umb_time_K}
	Let $(M,p) \in \mathcal{G}_2$ is generic. Then
	\begin{description}
		\item[(1)] The branches of the $LPL$, the $PC$ and the $MCNC$ curves are parwise transverse when they are real. There are three generic configurations as shown in Figure \ref{fig:desd_umb_time_K} middle figures.
		\item[(2)] For a generic 1-parameter family of surfaces $M_t$, with $M_0 = M$, the configurations of the three curves deform as in Figure \ref{fig:desd_umb_time_K}.
	\end{description}
			\begin{figure}[h!]
		\centering
		\begin{subfigure}[b]{0.5\textwidth}
			\includegraphics[width=\textwidth]{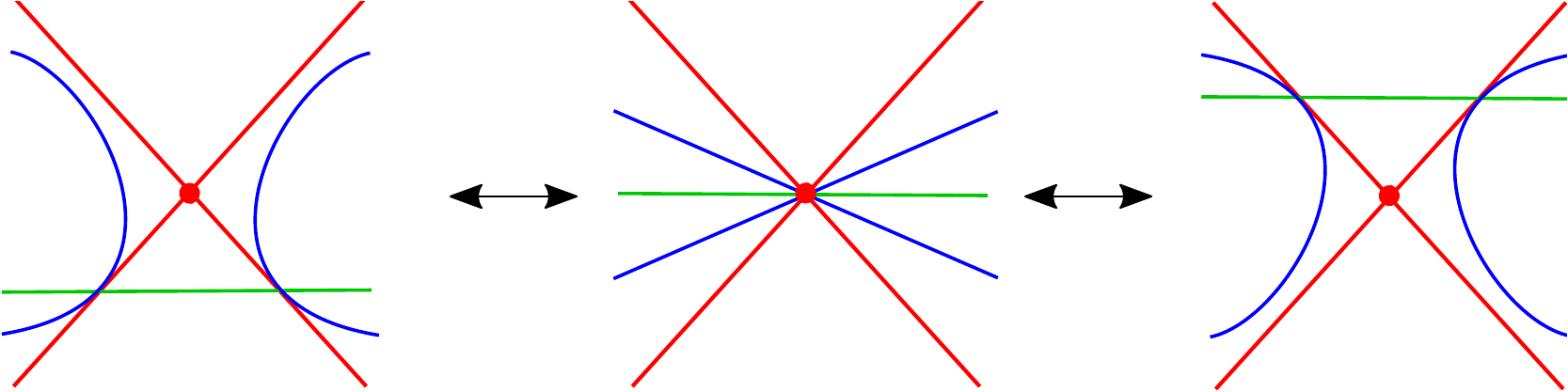}
		\end{subfigure}
		\quad\quad
		\begin{subfigure}[b]{0.5\textwidth}
			\includegraphics[width=\textwidth]{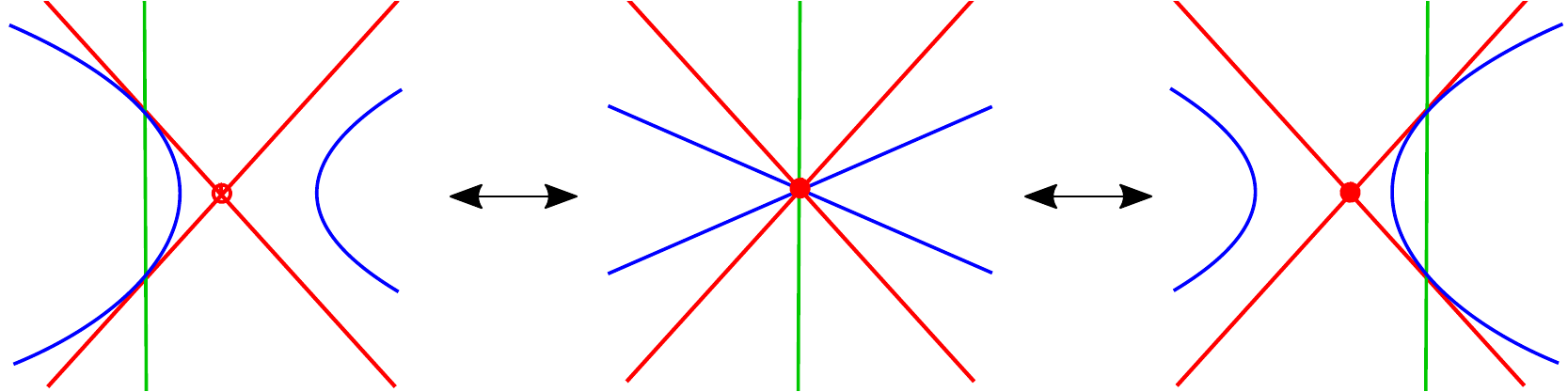}
		\end{subfigure}
		\quad
		\begin{subfigure}[b]{0.5\textwidth}
			\includegraphics[width=\textwidth]{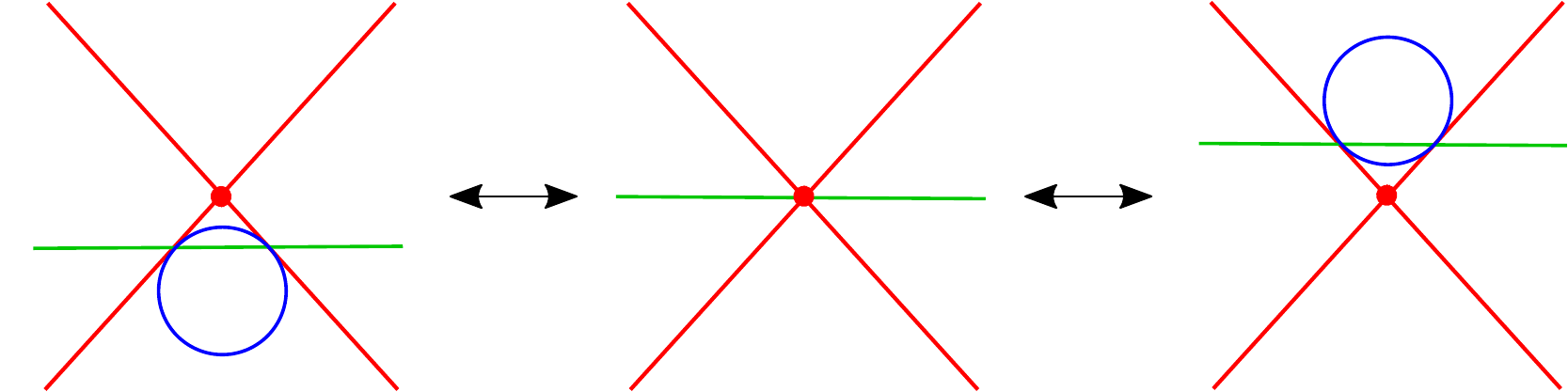}
		\end{subfigure}
		\caption{Generic bifurcations in the $LPL$ (red), the $PC$ (blue) and the $MCNC$ (green) at a flat timelike umbilic point.}
		\label{fig:desd_umb_time_K}
	\end{figure}
\end{theo}

\begin{proof}
	\textbf{(1)} It follows from $j^3f = a_{30} x^3+a_{31} x^2z+a_{32} xz^2+a_{33} z^3$ that the $LPL$ has a singularity of type $A_1^-$ at $p$ if and only if $\Lambda_3 = -a_{31}^2 + 3 a_{30} a_{32} + a_{32}^2 - 3 a_{31} a_{33} \neq 0$. The $PC$ has a singularity of type $A_1^\pm$ at $p$ if and only if $\Lambda_4 = -a_{31}^2 a_{32}^2 + 4 a_{30} a_{32}^3 + 4 a_{31}^3 a_{33} - 18 a_{30} a_{31} a_{32} a_{33} + 27 a_{30}^2 a_{33}^2 \neq 0$. The $MCNC$ is a regular curve at $p$ if and only if $\Lambda_5 = (a_{31} - 3 a_{33})^2+(a_{32} - 3 a_{30})^2 \neq 0$.
	
	The tangent lines to the $LPL$ at $p$ are $r_1 : c_1 x + d_1 z = 0$ and $r_2 : c_2 x + d_2 z = 0$, where $c_1 = 3 a_{30} - 2 a_{31} + a_{32}$, $d_1 = a_{31}-2 a_{32}+3 a_{33}$, $c_2 = 3 a_{30} + 2 a_{31} + a_{32}$ and $d_2 = a_{31} + 2 a_{32} + 3 a_{33}$.
	Also, the tangent line to the $MCNC$ at $p$ is $r_3: c_3 x + d_3 z = 0$, with $c_3 = 3 a_{30} - a_{32}$ and $d_3 = a_{31} - 3 a_{33}$. Thus, the condition for $r_3$ to be distinct from $r_1$ and $r_2$ are, respectively,
	$$ \Lambda_6 = a_{31}^2 - 3 a_{30} a_{32} - a_{31} a_{32} + a_{32}^2 + 9 a_{30} a_{33} - 3 a_{31} a_{33} \neq 0,$$
	$$\Lambda_7 = a_{31}^2 - 3 a_{30} a_{32} + a_{31} a_{32} + a_{32}^2 - 9 a_{30} a_{33} - 3 a_{31} a_{33} \neq 0.$$
	
	If $\Lambda_4 < 0$, the $PC$ has a singularity $A_1^+$ at $p$. Therefore, when $\Lambda_6 \neq 0$ and $\Lambda_7 \neq 0$, the only possibility is shown in Figure \ref{fig:desd_umb_time_K} middle last figure.
	
	When $\Lambda_4 > 0$, the $PC$ has a singularity $A_1^-$ and divides $U$ into four regions in a neighborhood of the origin. Suppose $\Lambda_6 \Lambda_7 \neq 0$. 
	Let $v_1$ and $v_2$ be the tangent vectors to $r_1$ and $r_2$, respectively, and $F(x,z) = j^2 \bar{K} = k_{20} x^2 +k_{21} x z+ k_{22} z^2$, with
	$$k_{20} = 4 (a_{31}^2 - 3 a_{30} a_{32}),\quad	k_{21} = 4 (a_{31} a_{32} - 9 a_{30} a_{33}),\quad k_{22} = 4 (a_{32}^2 - 3 a_{31} a_{33}).$$
	The relative position of the $LPL$ and of the $PC$ depends on the sign of $F(v_1) = \Lambda_6^2>0$ and $F(v_2)=\Lambda_7^2>0$. Thus, the branches of the $LPL$ are in the same region delimited by the $PC$.
	
	Finally, to the position of the $MCNC$, we define $G(x,z) = (c_1 x+d_1 z) (c_2 x+d_2 z)$. Note that $G(x,z)=0$ corresponds to the lines $r_1$ and $r_2$. As $F(v_1) = \Lambda_6^2 > 0$ and $F(v_2) = \Lambda_7^2 > 0$, there are two possible combinations for the sign of $F$ and $G$, as shown in Figure \ref {fig:sinal_F_G}.
	\begin{figure}[h!]
		\centering
		\includegraphics[width=0.5 \textwidth]{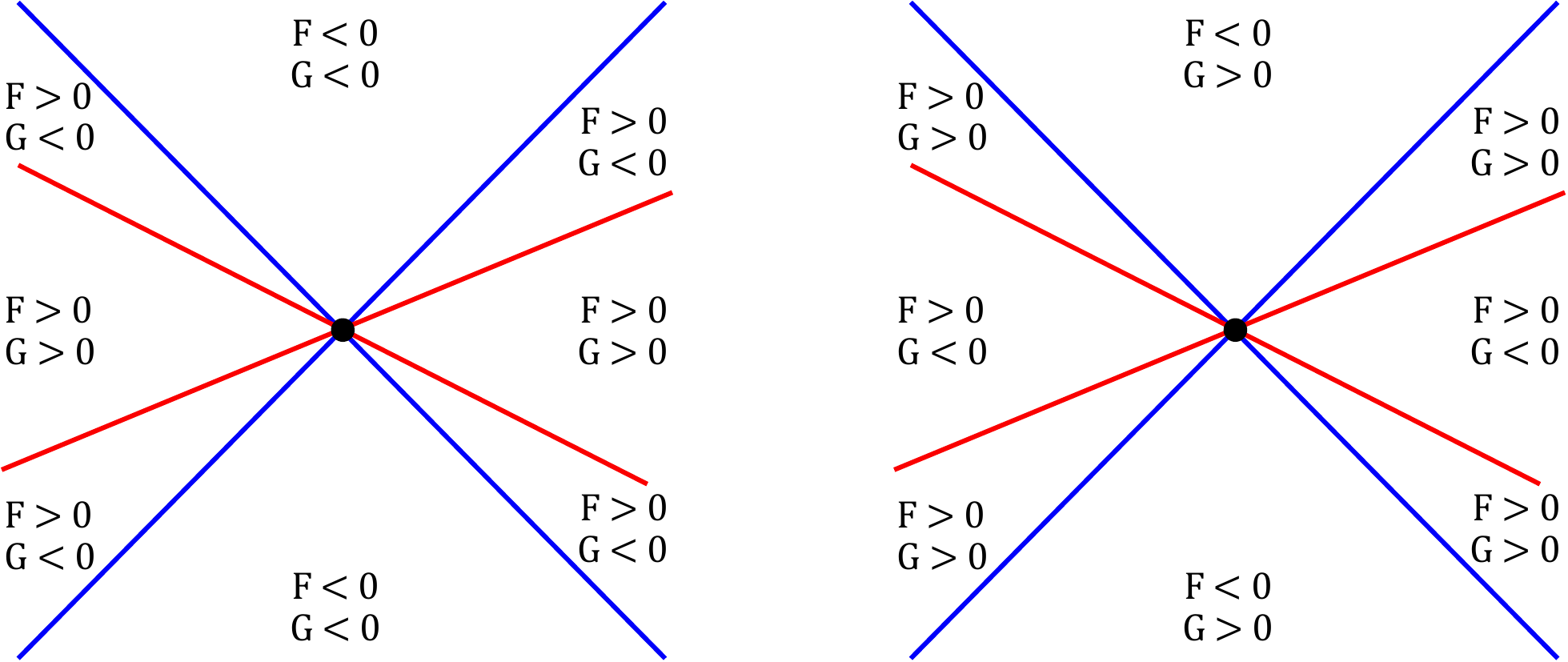}
		\caption{Possible combinations for the signs of $F$ and $G$. The blue lines are $F=0$ and the red lines are $G=0$.}
		\label{fig:sinal_F_G}
	\end{figure}
	
	Let $v_3$ be the tangent vector to the $MCNC$ at $p$. We have $F(v_3) = \Lambda_6 \Lambda_7$ and $G(v_3) = -4\Lambda_6 \Lambda_7$. If $F(v_3) < 0$, then $G(v_3) > 0$. Therefore, the configuration in Figure \ref{fig:sinal_F_G} right figure is the only one possible position when we consider the $MCNC$.
	
	For a generic surface $(M,p) \in \mathcal{G}_2$, all the conditions above are satisfied. Therefore, the configurations of the curves as in Figure \ref{fig:desd_umb_time_K} middle figures.

	\textbf{(2)} Let $M_t \subset \R^3_1$ be a family of surfaces parameterized by $\bx_t : U \to M_t$, with $\bx_t(x,z) = (x,h(x,z,t),z)$ for some $h$ differentiable with $h(x,z,0) = f(x,z)$.
	
	The $MCNC$ of $M_t$ are given by $\tilde \delta_t=0$ and $\bar H_t=0$, respectively. Since $\tilde \delta_t$ has an $A_1^-$-singularity for every $t$, we have $j^2\tilde \delta_t=f_1(x,z,t) f_2(x,z,t)$,
	where $f_1, f_2: \R^3 \to \R$ are degree 1 polynomials in $x$ and $y$, with
	$$\det \left( \begin{array}{cc}
	\frac{\partial f_1}{\partial x}(0,0,0) & \frac{\partial f_1}{\partial y}(0,0,0) \vspace{0.2cm} \\
	\frac{\partial f_2}{\partial x}(0,0,0) & \frac{\partial f_2}{\partial y}(0,0,0)
	\end{array} \right) \neq 0.$$
	
	From the Inverse Function Theorem, it follows that $X = f_1(x,z,t)$, $Z = f_2(x,z,t)$, $T = t$ is a change of coordinates in a neighborhood of the origin. Applying this coordinate change to $\tilde \delta_t$ and $\bar H_t$, we get $\tilde \delta_t \sim_\mathcal{R} XZ$ and $\bar H_t \sim_\mathcal{R} \frac{\Lambda_6}{2 d_1 d_2} X+\frac{\Lambda_7}{2 d_1 d_2} Z+\frac{2 \Lambda_8}{d_1 d_2} T$,	where
	$$\Lambda_8 = -3 a_{30} a_{32} h_{zzt}+9 a_{30} a_{33} h_{xzt}+a_{31}^2 h_{zzt}-a_{31} a_{32} h_{xzt}-3 a_{31} a_{33} h_{xxt}+a_{32}^2 h_{xxt} \neq 0,$$
	and the derivatives of $h$ are evaluated in $(0,0,0)$. Since $\Lambda_6 \neq 0$, $\Lambda_7 \neq 0$ and $\Lambda_8 \neq 0$ are generic conditions, in generic 1-parameter family of surfaces the $LPL$ and the $MCNC$ deforms as in Figure \ref{fig:lpl_H} at a flat timelike umbilic point.
	
	\begin{figure}[h!]
		\centering
		\includegraphics[width=0.5 \textwidth]{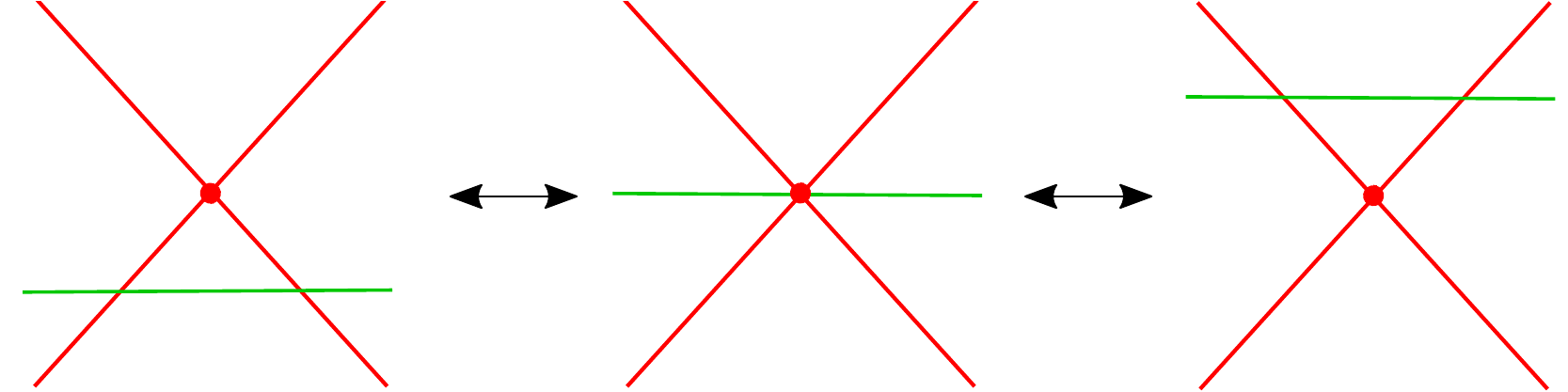}
		\caption{Deformations of the $LPL$ (red) and the $MCNC$ (green) in a generic family of surfaces at a flat timelike umbilic point.}
		\label{fig:lpl_H}	
	\end{figure}
	
	Let $\bar K_t$ be such that the $PC_t$ is defined by $\bar K_t = 0$. We can assume that $\Lambda_4 \neq 0$ and $k_{20} \neq 0$. We have
	$$j^2\bar K_t \sim_{\mathcal{R}^{(2)}} \pm \left(\pm x^2 + z^2 - \frac{\Lambda_8^2 t^2}{12 |k_{20} \Lambda_4|}\right)+t^3 h(t).$$
	and the bifurcations in the $PC$ in generic 1-parameter family of surfaces are generic cone sections as in Figure \ref{fig:desd_umb_time_K}. Therefore, the result follows from Theorem \ref{prop:inter_tripla}.
\end{proof}

\subsection{Bifurcations at a none Morse singularity of the LPL}\label{sec:umb_A3}

At a timelike umbilic point, the $LPL$ has generically an $A_1^-$-singularity. We consider here the case where the $LPL$ has a more degenerate singularity.
%
%
%
%
Let 
$$\mathcal{G}_3 = \{ (M,p) \in \mathcal{G}: p {\rm \,\,is \,\,a \,\,non-Morse \,\,singularity \,\,of \,\,the \,\,} LPL \}.$$
Given $(M,p) \in \mathcal{G}_3$, take $\bx : U \subset \R ^2 \to \R^3_1$ the local parameterization of $M$ given by Theorem \ref{the:param_lightlike}. 
The $LPL$ in $U$ is given by $\tilde \delta (u,v) = \bar l (u,v)\bar n(u,v) = 0$, with
$j^2 \bar l = \sum_{i=1}^{2} \sum_{j=0}^{i} a_{ij} x^{i-j} y^j$ and $j^2 \bar n = \sum_{i=1}^{2} \sum_{j=0}^{i} b_{ij} x^{i-j} y^j$.
The $LPL$ has an $A_1^-$-singularity at the origin when $a_{10} b_{11}-a_{11} b_{10} \neq 0$.

\begin{theo}\label{teo:umb_time_A3}
	For a generic $(M,p) \in \mathcal{G}_3$, we have
	\begin{description}
		\item[(1)] The $LPL$ has a $A^-_3$ singularity at $p$;
		\item[(2)] The bifurcation of the $LPL$ in a generic 1-parameter family of surfaces $M_t$, with $M_0=M$, are as in Figure \ref{fig:umb_time_A3}.
	\end{description}
	\begin{figure}[h!]
		\centering
		\includegraphics[scale=0.4]{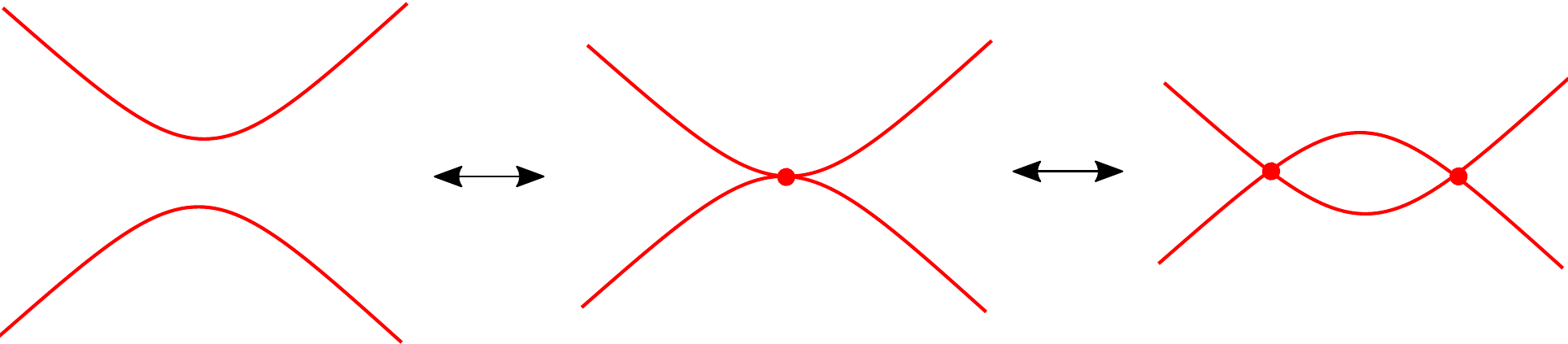}
		\caption{Bifurcations of the $A^-_3$ singularity of the $LPL$ on a generic 1-parameter family of surfaces.}
		\label{fig:umb_time_A3}
	\end{figure}
\end{theo}

\begin{proof}
	\textbf{(1)} We have $a_{10} b_{11} - a_{11} b_{10} = 0$ as the $LPL$ has singularity more degenerate than $A_1^-$. Suppose that $a_{10} \neq 0$ or $b_{10} \neq 0$. Then $\tilde \delta \sim_{\mathcal{R}^{(4)}} \pm (u^2-\Lambda_9^2 v^4)$,
	with
	$$\Lambda_9 = -a_{11}^2 a_{20} b_{10} + a_{10} a_{11} a_{21} b_{10} - a_{10}^2 a_{22} b_{10} + a_{10} a_{11}^2 b_{20} - a_{10}^2 a_{11} b_{21} + a_{10}^3 b_{22}.$$
	Generically, $\Lambda_9 \neq 0$ and the singularity is of type $A_3^-$.

	\noindent \textbf{(2)} We take $\bx_t : U \to \R^3_1$ a parameterization of $M_t$ given by Theorem \ref{the:param_lightlike} which depends smoothly on $t \in (\R,0)$ and $\bx_0 = \bx$. Denote by $E(u,v,t)$, $F(u,v,t)$, $G(u,v,t)$, $\bar l(u,v,t)$, $\bar m(u,v,t)$ and $\bar n(u,v,t)$ the coefficients of the first and second fundamental forms of $M_t$ in $(u,v) \in U$ with respect to $ \bx_t$. Thus, $E(u,v,t) = G(u,v,t) = 0$ on $U$ and $\tilde \delta_t (u,v) = \bar l (u,v,t) \bar n(u,v,t)$.

	The deformation of the $A_3^-$ singularity of the $LPL$ (see \S \ref{sec:appendix}) induced by the family $M_t$ satisfy
	$$\tilde \delta_t(u,v) \sim_{\mathcal{R}} \pm (x^2-\Lambda_9^2 y^4 + \gamma_{1}(t) y^2+\gamma_{2}(t)y+\gamma_{3}(t)),$$
	where $\gamma_{1}$, $\gamma_{2}$ and $\gamma_{3}$ are germs of differentiable functions. We get a germ of a differentiable curve $\gamma (t) = (\gamma_{1}(t),\gamma_{2}(t),\gamma_{3}(t))$ in $\R^ 3$ with $\gamma'(0)=\left(\frac{\Lambda_9}{a_{10} b_{10}},0,0\right)$.
	
	A point $q = \bx_t(u,v) \in M_t$ is a singularity of the $LPL$ if and only if $q$ is an umbilic point, that is, $\sigma_t(u,v) = (0,0)$, where $\sigma_t(u,v) = \left(\bar{l} (u,v,t),\bar{n} (u,v,t)\right)$. We have
	$$\sigma_t(u,v) \sim_{\mathcal{A}} \left(u,\frac{a_{10}^2 (a_{10} \bar n_t(0,0,0) -b_{10} \bar l_t(0,0,0)) t}{\Lambda_9} + v^2 \right),$$
	provided that $a_{10}^2 (a_{10} \bar n_t(0,0,0) -b_{10} \bar l_t(0,0,0)) \neq 0$. Therefore, $\sigma_t$ can have either $0$ or $2$ zeros when $t \neq 0$, depending on the sign of $A = \Lambda_9(a_{10} \bar n_t(0,0,0) -b_{10} \bar l_t(0,0,0)) t$.
		
	When $A<0$, the only possibility is that $\gamma$ belongs to the stratum $(IV)$ of Figure \ref{fig:desd_A3} left figure, because there are two singularities. On the other hand, when $A>0$, $\gamma$ must be in the $(IX)$ or $(X)$ stratum, since $\gamma$ is regular. Therefore, the deformation of the $LPL$ with singularity $A_3^-$ is as shown in Figure \ref{fig:umb_time_A3}.
\end{proof}

\subsection{Bifurcations at a singularity of the PC and MCNC}\label{sec:sing_pc_mcnc}

Regular points on the $PC$ and $MCNC$ have codimension 0. Therefore, in this subsection, we will study the singularities of these curves at points of codimension 1.

Consider $MCNC$ in Lorentzian region, the Riemannian case is analogous. Let 
$$\mathcal{G}_4 = \{(M,p)\in\mathcal{G}:p {\rm \,\,is \,\,a \,\,singularity \,\,of \,\,the \,\,}MCNC{\rm \,\,in \,\,Lorentzian \,\,region}\}.$$

\begin{theo}\label{teo:sing_mcnc}
	For a generic $(M,p) \in \mathcal{G}_4$, we have
	\begin{description}
		\item[(1)] The $MCNC$ has a $A_1$-singularity at $p$;
		\item[(2)] In generic 1-parameter family of surfaces, the $A_1$-singularity of the $MCNC$ at $p$ is $\mathcal{R}$-versally deformed.
	\end{description}
\end{theo}

\begin{proof}
	\textbf{(1)} Given $(M,p) \in \mathcal{G}_4$, take $\bx : U \subset \R ^2 \to \R^3_1$ the local parameterization of $M$ given as in (\ref{eq:param_monge_time}) by $\bx(x,z) = (x,f(x,z),z)$ with $p = \bx(0,0)$. Since $p$ is a singularity of $MCNC$, we have $a_{22}=a_{20}$, $a_{32}=3 a_{30}$ and $a_{31}=3 a_{33}$. This is a Morse singularity when $\Lambda_{10} \neq$ with
	$$\begin{array}{ccl}
	\Lambda_{10} & = & -16 \left(4 a_{20}^3-a_{20} a_{21}^2-6 a_{40}+a_{42}\right) \left(4 a_{20}^3-a_{20} a_{21}^2-a_{42}+6 a_{44}\right)\\
	&&+\left(8 a_{20}^2 a_{21}-2 a_{21}^3-6 a_{41}+6 a_{43}\right)^2.
	\end{array}$$
		
	\noindent \textbf{(2)} Let $M_t \subset \R^3_1$ be a 1-parameter family of surfaces parameterized by $\bx_t : U \to M_t$, with $\bx_t(x,y) = (x,h(x,z,t),z)$ for some $h$ differentiable with $h(x,z,0) = f(x,z)$ and $f$ as in item (1). The $MCNC_t$ is formed by the points $\bx_t(x,y)$ where $\bar H_t(x,y) = 0$. Suppose $\Lambda_{11} = 4 a_{20}^3-a_{20} a_{21}^2-6 a_{40}+a_{42} \neq 0$. We have $j^2\tilde H_t \sim_{\mathcal{R}^{(2)}} \rho(t) + x^2 + \frac{\Lambda_{10}}{8 \Lambda_{11}} y^2$ and $j^1\rho = (h_{xx}(0,0,0)-h_{yy}(0,0,0)) t$. Therefore, the $MCNC_t$ is a $\mathcal{R}-$versal deformation of the $MCNC$ of $M$.
\end{proof}

The $PC$ of a surface is independent of the metric because it can be defined using the surface contact with the tangent plane \cite{GeoGen}. Therefore, the results obtained in the Lorentzian region can be extended to the Riemannian region. Let 
$$\mathcal{G}_5 = \{(M,p) \in \mathcal{G} :p {\rm \,\,is \,\,a \,\,singularity \,\,of \,\,the \,\,}PC{\rm\,\, in \,\,Lorentzian \,\,region}\}.$$

\begin{theo}\label{teo:sing_pc}
	For a generic $(M,p) \in \mathcal{G}_5$, we have
	\begin{description}
		\item[(1)] The $PC$ has a $A_1$-singularity at $p$;
		\item[(2)] In generic 1-parameter family of surfaces, the $A_1$-singularity of the $PC$ at $p$ is $\mathcal{R}$-versally deformed.
	\end{description}
\end{theo}

\begin{proof}
	Similar to the proof of the Theorem \ref{teo:sing_mcnc}.
\end{proof}

\section{Bifurcations at points on the LD}\label{sec:ld}

\subsection{Bifurcations at $\textbf{LD} \cap \textbf{LPL} \cap \textbf{MCNC}$}\label{sec:int_ld_lpl}

In this section, we study the intersections between the $LD$ and the $LPL$ in the case where these curves are regular (singular case is considered in \S \ref{sec:umb_light}). It is show next that $LD \cap LPL$ coincide with $LD \cap MCNC$. Therefore, every point in $LD \cap LPL$ is also in $MCNC$.

\begin{theo}\label{teo:ld_lpl_H}
	Let $M \subset \R^3_1$ be a smooth surface and $p \in LD$. Then $p \in LPL$ if and only if $p \in MCNC$. Furthemore, if $p$ is a regular point on the $LD$, then $m(LD,LPL:p) = 2 m(LD,MCNC:p)$. Consequently, $m(LD,LPL:p)$ is always even.
\end{theo}

\begin{proof}
	We take $\bx : U \to \R^3_1$ a local parameterization of $M$ as in Theorem \ref{the:param_lightlike} with $p = \bx(0,0)$. Thus, $\bar H (0,0) = \bar l(0,0) G(0,0)$ and $\tilde \delta(0,0) = l(0,0)^2 G(0,0)^2$. Therefore, $\bar H(0,0) = 0$ if and only if $\tilde \delta(0,0) = 0$.
	
	Let $\gamma: (\R,0) \to (\R^2,0)$ be a parametrization of the $LD$ at a regular point $p$. Since $E = F = 0$ along the $LD$, it follows that $E \circ \gamma$ and $F \circ \gamma$ are identically zero. Therefore,
	$$\bar H (\gamma (x)) = (\bar l G-2\bar mF+\bar n E) (\gamma(x)) = \bar l (\gamma(x)) G (\gamma (x)),$$
	$$\tilde \delta(\gamma(x)) = ((\bar lG-\bar nE)^2-4(\bar mG-\bar nF)(\bar lF-\bar mE))(\gamma (x)) = \left( \bar l (\gamma(x)) G (\gamma(x)) \right)^2$$
	and the result follows.
\end{proof}


Points where the $LD$ and the $MCNC$ are transversal are stable.
%
Let 
$$\mathcal{G}_6 = \{(M,p) \in \mathcal{G}:m(LD,LPL:p)>2\}.$$
Therefore, the $LD$ and the $MCNC$ of $M$ are tangent at $p$ for every $(M,p) \in \mathcal{G}_6$. Given $(M,p) \in \mathcal{G}_6$, take $\bx : U \to \R^3_1$ the local parameterization of $M$ as in (\ref{eq:param_monge_light}) with $p = \bx(0,0)$. As $p \in LPL$ and is not an umbilic point, $a_{20} = 0$ and $a_{21} \neq 0$. Applying a homothety, consider $a_{21} = 1$. As $m(LD,LPL:p)>2$, it follows that $a_{30} = -\frac{1}{3}$. 


%
%

\begin{theo}\label{teo:flat_umb}
	Let $(M,p) \in \mathcal{G}_6$ is generic. Then 
	\begin{description}
		\item[(1)] $m(LD,LPL:p)=4$ and $m(LD,LPL:MCNC)=2$. There are two possible
		configurations for the three curves as shows in Figure \ref{fig:ld_lpl_cmn_def} middle figures.
		\item[(2)] For a generic 1-parameter family of surfaces $M_t$, with $M_0=M$, the configurations of the three curves deform as in Figure \ref{fig:ld_lpl_cmn_def}.
	\end{description}
	\begin{figure}[h!]
		\centering
		\includegraphics[width=0.9\textwidth]{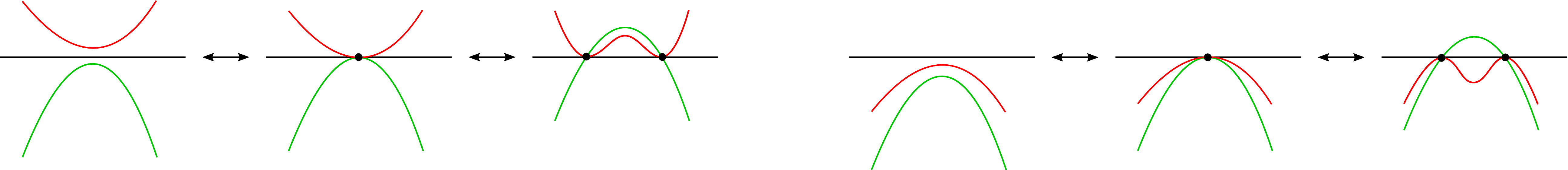}
		\caption{Deformations of the $LD$ (black), the $LPL$ (red) and the $MCNC$ (green) in a generic 1-parameter family of surfaces at points $p$ where $m(LD,LPL) = 4$.}
		\label{fig:ld_lpl_cmn_def}
	\end{figure}
\end{theo}

\begin{proof}
	\textbf{(1)} Similar to the proof of Theorem \ref{teo:lpl_k_mcn} (1). \vspace{0.2cm}

	\noindent \textbf{(2)} Let $M_t \subset \R^3_1$ be a family of surfaces parameterized by $\bx_t : U \to M_t$, with $\bx_t(x,z) = (x,y,h(x,y,t))$	for some $h$ differentiable with $h(x,y,0) = f(x,y)$. Thus, the $LD_t$ is given by $\delta_t(x,y) = 0$, with $\delta_t(x,y) = h_x(x,y,t)^2+h_y(x,y,t)^2-1$. For $t=0$, we have that
	$$\frac{\partial \delta_0}{\partial y}(0,0) = 2 \neq 0.$$
	So, for $t$ small enough, the $LD_t$ is a regular curve and there is $y_t : (\R,0) \to (\R,0)$ such that $\delta_t(x,y_t(x)) = 0$, for every $x$ close to $0$.	If $g_t(x) = \bar H_t(x,y_t(x))$, then
	$$g_t \sim_\mathcal{R} \Lambda_{12} x^2+\left( 2 h_{yt}(0,0,0)-2 (2 a_{22}+ a_{31}) h_{xt}(0,0,0)+h_{xxt}(0,0,0) \right) t + \rho (t),$$
	with $j^1\rho = 0$ and $\Lambda_{12} = 4 a_{22}+7 a_{31}+12 a_{40}$. Therefore, the $MCNC_t$ and the $LD_t$ intersect at 0 or 2 points, depending on the sign of $t$. As $\bar K(0,0) \neq 0$, since $p$ does not belong to the $PC$, the $LPL_t$ does not intersect with the $PC_t$ for small values of $t$, that is, the only intersections of the $LPL_t$ and the $MCNC_t$ occur along the $LD_t$.	
\end{proof}

\subsection{Bifurcations at $\textbf{LD} \cap \textbf{PC}$}\label{sec:int_ld_K}

The intersections between the $LD$ and the $PC$ are the simplest cases dealt with in this papper. The transversal intersections between these curves have codimension 0. We consider here the case where the $LD$ and the $PC$ are tangent. So, let
$$\mathcal{G}_7 = \{(M,p) \in \mathcal{G}:m(LD,PC:p)>1\}.$$
%

\begin{theo}
	Let $(M,p) \in \mathcal{G}_7$ be generic. Then
	\begin{description}
		\item[(1)] $m(LD,PC:p) = 2$.
		\item[(2)] For a generic 1-parameter family of surfaces $M_t$, with $M_0=M$, the configurations of the two curves deform as in Figure \ref{fig:def_ld_K}.
	\end{description}
	\begin{figure}[h!]
		\centering
		\includegraphics[width=0.5\textwidth]{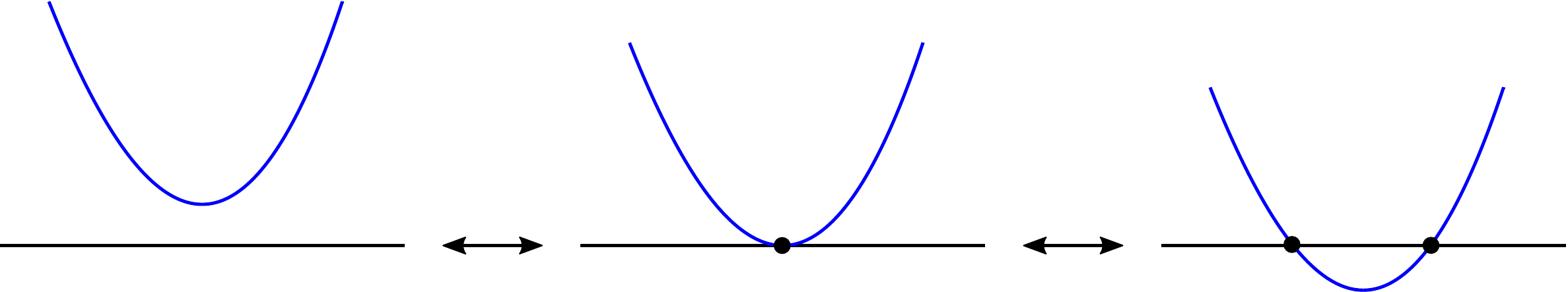}
		\caption{Deformations of the $LD$ (black) and the $PC$ (blue) in a generic family of surfaces.}
		\label{fig:def_ld_K}
	\end{figure}
\end{theo}

\begin{proof}
	\textbf{(1)} Similar to the proof of Theorem \ref{teo:lpl_k_mcn} (1). \textbf{(2)} Similar to the proof of the Theorem \ref{teo:flat_umb} (2).
\end{proof}

\subsection{Bifurcations at a lightlike umbilic point}\label{sec:umb_light}

Lightlike umbilic points are singularities of the $LD$ \cite{CaratheodoryR31}. We prove below that the $LPL$ (extended to the $LD$ as the discriminant of equation (\ref{eq:principalLD})) passes through lightlike umbilic points and is also singular at that point.


\begin{prop}
	Let $M$ be a smooth surface in $\R^3_1$ and $p \in LD$.
	\begin{description}
		\item[(1)] $p$ is a lightlike umbilic point if and only if the $LPL$ is singular at $p$
		\item[(2)] If $p$ is a lightlike umbilic point, then $p$ belongs to the $PC$.
		\item[(3)] If $p$ belongs to the $LPL$ and the $PC$, then $p$ is an umbilic point.
		\item[(4)] $p$ is a lightlike umbilic point if and only if $p$ belongs to the $PC$ and the $MCNC$.		
	\end{description}
\end{prop}

\begin{proof}
	Consider $\bx : U \to \R^3_1$ the local parameterization of $M$ given by the Theorem \ref{the:param_lightlike} with $p = \bx(q)$. Thus, $E(q) = F(q) = 0$, $E_u(q) = \lambda \bar l(q)$ and $E_v(q) = \lambda \bar m(q)$, for some $\lambda \neq 0$. The curvature lines BDE is $\bar m(q) dv^2+\bar l(q) dv du = 0$.\vspace{0.3cm}
	
	\noindent \textbf{(1)} It follows that
	$$\begin{array}{lcl}
	\tilde \delta(q) & = & G(q) \bar l(q),\vspace{0.2cm}\\
	\tilde \delta_u(q) & = & 2 G(q) \bar l(q) \left(-2 F_u(q) \bar m(q)+G_u(q) \bar l(q)+G(q) \bar l_u(q)-\lambda  \bar l(q) \bar n(q)+2 \lambda  \bar m(q)^2\right),\vspace{0.2cm}\\
	\tilde \delta_v(q) & = & 2 G(q) \left(\bar l(q) \left(G(q)  \bar l_v(q)-\bar m(q) (2 F_v(q) +\lambda \bar n(q))+G_v(q) \bar l(q)\right)+2 \lambda \bar m(q)^3\right).
	\end{array}$$
	Since $G(q) \neq 0$, $p$ is a singular point of the $LPL$ if and only if $\bar l(q) = \bar m(q) = 0$.\vspace{0.3cm}
	
	\noindent \textbf{(2)} If $p$ is a umbilic point, then $\bar l (q) = \bar m(q) = 0$. Therefore, $\bar K(q) = 0$ and $p$ belongs to the $PC$.\vspace{0.3cm}
	
	\noindent \textbf{(3)} If $p \in LPL$, then $\bar l(q) = 0$. Thus, $\bar K(q) = \bar m(q)^2$. As $p$ belongs to the $PC$, it follows that $\bar K(q) = 0$. Therefore, $\bar m (q) = 0$ and $p$ is a umbilic point.\vspace{0.3cm}
	
	\noindent \textbf{(4)} Follows from items 2, 3 and the Theorem \ref{prop:inter_tripla}.
\end{proof}

%
%

Let 
$$\mathcal{G}_8 = \{(M,p)\in \mathcal{G}:p {\rm \,\,is \,\,a \,\,lightlike \,\,umbilic \,\,point}\}.$$
For a $(M,p) \in \mathcal{G}_8$, we use the local parameterization $\bx: U \to \R^3_1$ of $M$ given as in (\ref{eq:param_monge_light}) with $p = \bx(0,0)$ and $j^2f = x+a_{22} y^2$.

\begin{theo}\label{teo:umb_lightlike}
	Let $(M,p) \in \mathcal{G}_8$ is generic. Then
	\begin{description}
		\item[(1)] The $LD$ and the $LPL$ have $A_1^\pm$ and $A_3^\pm$ singularities at $p$, respectively, and the $PC$ and the $MCNC$ are regular curves. There are six generic configurations as shown in Figure \ref{fig:umb_tipo_luz_desd} middle figures.
		\item[(2)]  For a generic 1-parameter family of surfaces $M_t$ , with $M_0=M$, the configurations of the $LD$, $LPL$, $PC$ and $MCNC$ deform as in Figure \ref{fig:umb_tipo_luz_desd}
	\end{description}
	\begin{figure}[h!]
		\centering
		\begin{subfigure}[b]{0.4\textwidth}
			\includegraphics[width=\textwidth]{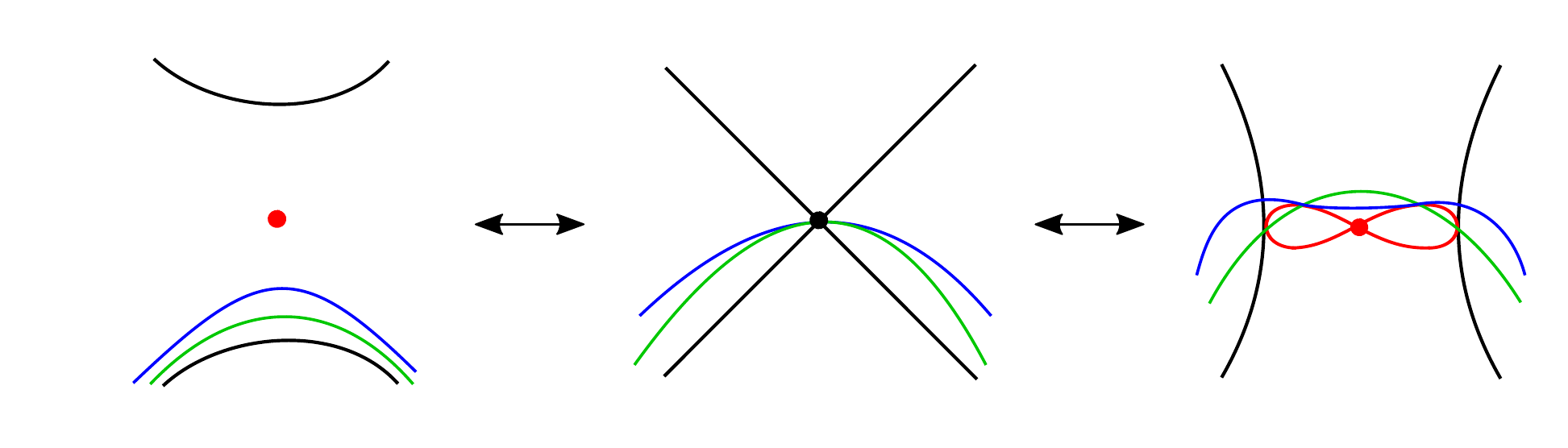}
		\end{subfigure}
		\qquad \quad
		\begin{subfigure}[b]{0.4\textwidth}
			\includegraphics[width=\textwidth]{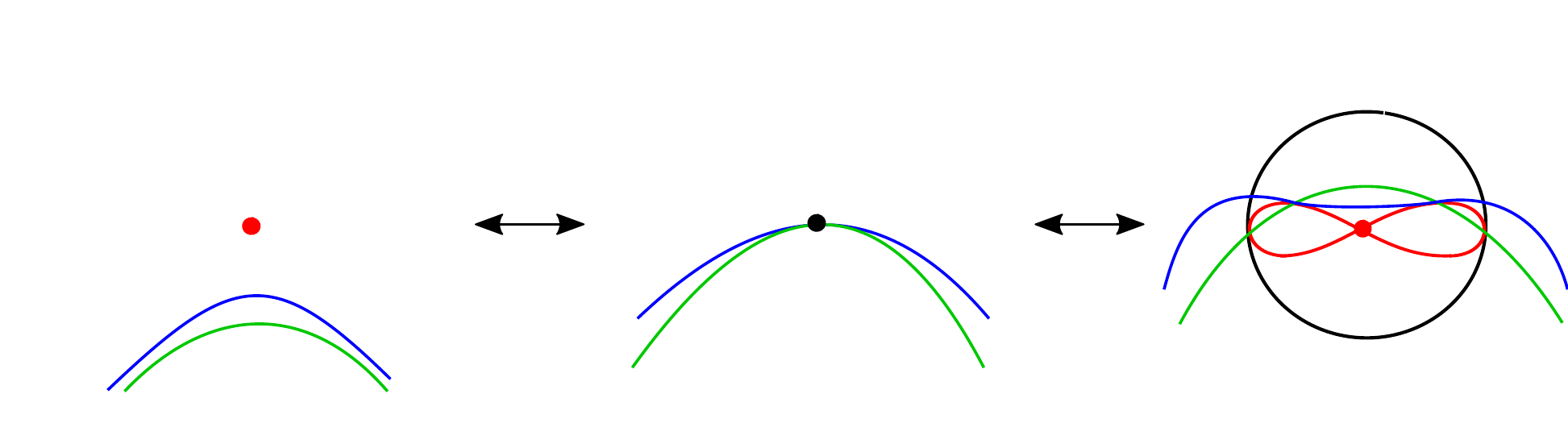}
		\end{subfigure}
		
		\begin{subfigure}[b]{0.4\textwidth}
			\includegraphics[width=\textwidth]{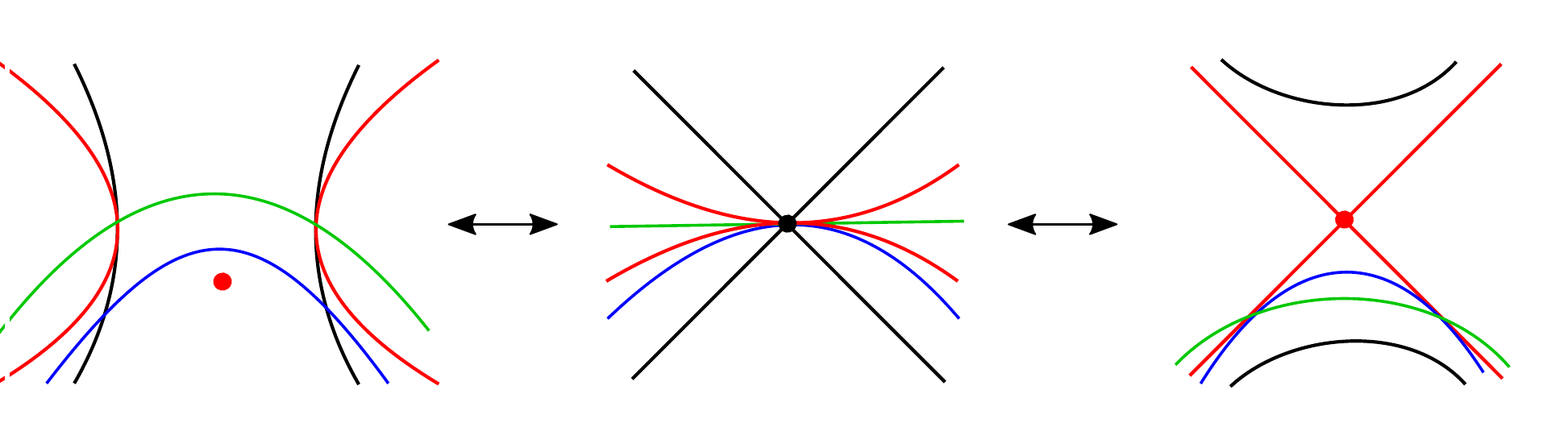}
		\end{subfigure}
		\qquad \quad
		\begin{subfigure}[b]{0.4\textwidth}
			\includegraphics[width=\textwidth]{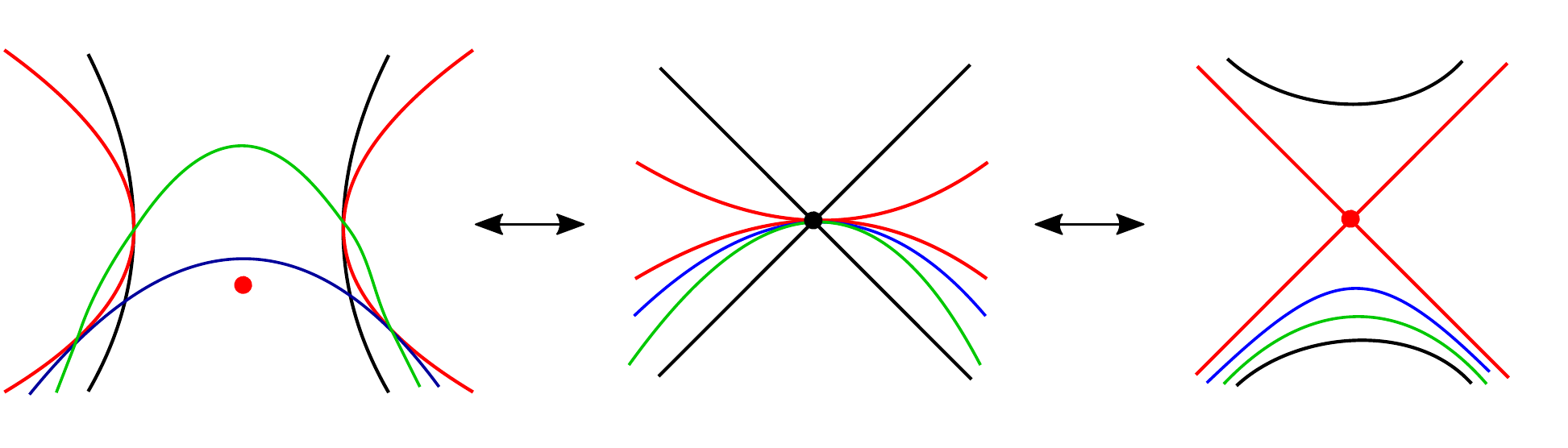}
		\end{subfigure}
		
		\begin{subfigure}[b]{0.4\textwidth}
			\includegraphics[width=\textwidth]{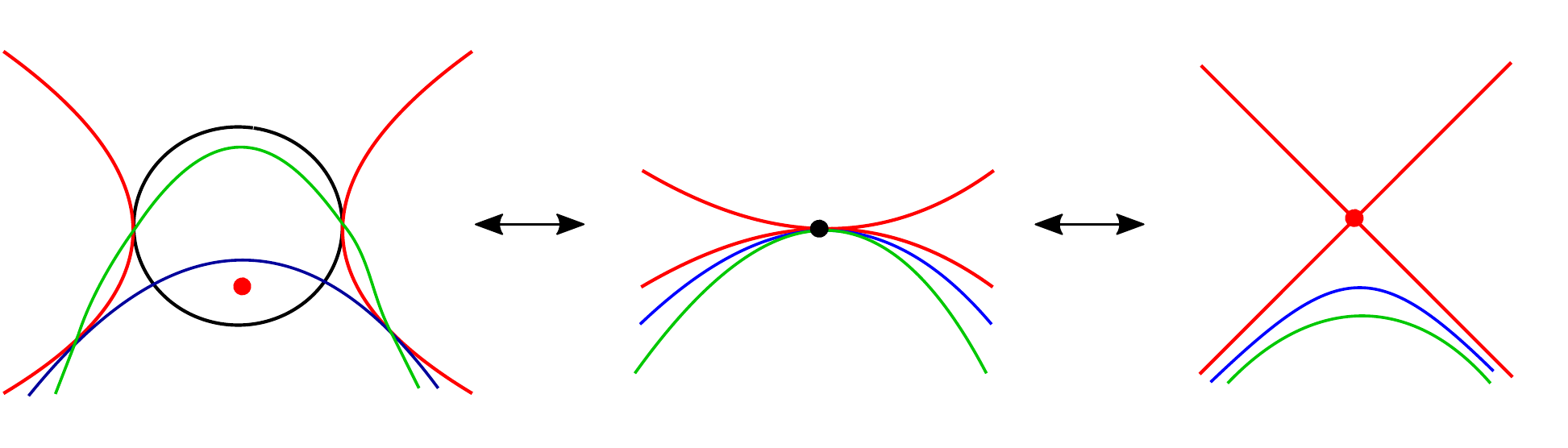}
		\end{subfigure}
		\qquad \quad
		\begin{subfigure}[b]{0.4\textwidth}
			\includegraphics[width=\textwidth]{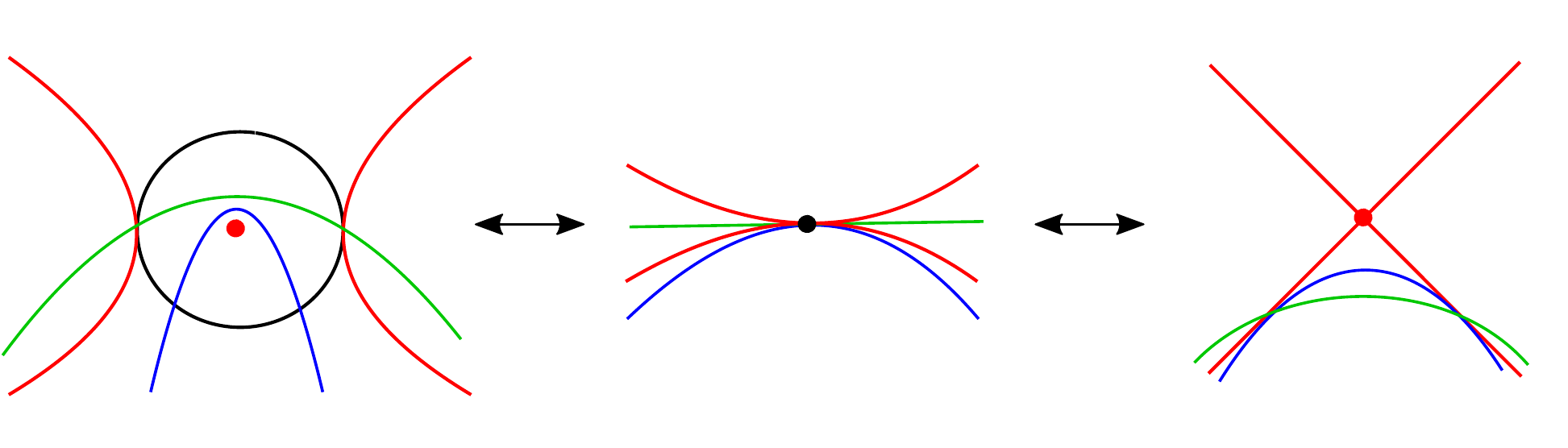}
		\end{subfigure}
		\caption{Deformation of the curves $LD$ (black), $LPL$ (red), $PC$ (blue) and $MCNC$ (green) in generic families of surfaces.}
		\label{fig:umb_tipo_luz_desd}
	\end{figure}
\end{theo}

\begin{proof}
	\textbf{(1)} Let $\delta = F^2-EG$, then $j^2\delta = 6 a_{30} x^2+4 a_{31} xy+(4 a_{22}^2+2 a_{32}) y^2$ and the $LD$ is given by $ \delta = 0$.	Thus, the $LD$ has a Morse singularity at $p$ if and only if $\Lambda_{13} = 6 a_{22}^2 a_{30}+3 a_{30} a_{32}-a_{31} ^2 \neq 0$. 
	
	On the other hand, 
	if $a_{30} \neq 0$, we have	$j^4\tilde{\delta} \sim_{\mathcal{R}^{(4)}} x^2-\frac{32 \Lambda_{13}^3}{27 a_{30}^3}y^4$. Therefore, the $LPL$ has a singularity $A_3^\pm$ at $p$. For the $PC$ and the $MCNC$, it suffices to note that $j^1 \bar{K} = 4a_{22}(3 a_{30} x+a_{31} y)$ and $j^1 \bar{H} = 2 (3 a_{30} x+a_{31} y)$. Suppose $\Lambda_{13} \neq 0$, $a_{30} \neq 0$ and $a_{22} \neq 0$. Note that the $LPL$, the $PC$ and the $MCNC$ are tangent at $p$ and the tangent line is $3 a_{30} x+a_{31} y = 0$.
	
	If the $LPL$ has a $A_3^+$ singularity, then we have two possibilities, one for each singularity of the $LD$ ($A_1^+$ or $A_1^-$).
	
	When the $LPL$ has a $A_3^-$ singularity, after coordinate changes it follows that 	\begin{equation}\label{eq:lpljato4}
	\tilde{\delta} \sim_{\mathcal{R}} x^2-\frac{32 \Lambda_{13}^3}{27 a_{30}^3}y^4 = \left(x - \sqrt{\frac{32 \Lambda_{13}^3}{27 a_{30}^3}} y^2\right) \left(x + \sqrt{\frac{32 \Lambda_{13}^3}{27 a_{30}^3}} y^2\right).
	\end{equation}
	Therefore, the $LPL$ is locally formed by the graphs of the functions $x = \pm \sqrt{\frac{32 \Lambda_{13}^3}{27 a_{30}^3}} y^2$.	Applying the same coordinate changes to $\bar K$ and $\bar H$, then
	$$\begin{array}{l}
	j^2\bar K = -2 a_{22} x+\frac{\Lambda_{13}}{9 a_{30}^2} x^2-\frac{2\left(-12 a_{22}^2 a_{30} a_{31}+2 a_{31}^3-9 a_{32} a_{30} a_{31}+27 a_{33} a_{30}^2\right)}{9 a_{30}^2} xy+\frac{4 \Lambda_{13} \left(6 a_{22} ^2 a_{30}+\Lambda_{13}\right)}{9 a_{30}^2 }y^2, \vspace{0.2 cm}\\
	j^2\bar H = x - \frac{2 a_{22}}{3 a_{30}} x^2 - \frac{8 a_{22} \Lambda_{13}}{3 a_{30}} y^2.
	\end{array}$$
	From the Implicit Function Theorem, there are differentiable functions $x_K, x_H: (\R,0) \to (\R,0)$ such that $\bar{K} (x_K(y),y) = 0$ and $\bar{H} (x_H(y),y) = 0$, for all $y$ close to $0$, with $x_K''(0)=  \frac{4 \Lambda_{13} \left(6 a_{22}^2 a_{30}+\Lambda_{13}\right)}{9 a_{22} a_{30}^2 }$ and $x_H''(0) = \frac{16 a_{22} \Lambda_{13}}{3 a_{30}}$.
	
	As the $LPL$, the $PC$ and the $MCNC$ are graphs of functions with 1-jet null, the position of these curves depends on the coefficient of $y^2$. First, we have that
	$$\left| \frac{x_K''(0)}{2} \right| > \sqrt{\frac{32 \Lambda_{13}^3}{27 a_{30}^3}} \quad \Leftrightarrow \quad 3 a_{30} a_{32}-a_{31}^2 \neq 0,$$
	and generically the $PC$ is never between the two branches of the $LPL$. 
	
	Furthemore, note that
	$$\left| \frac{x_H''(0)}{2} \right| > \sqrt{\frac{32 \Lambda_{13}^3}{27 a_{30}^3}} \quad \Leftrightarrow \quad \frac{3 a_{30} a_{32}-a_{31}^2}{a_{30}} < 0.$$
	Thus, the position between the $LPL$ and the $MCNC$ depends on the sign of $\frac{3 a_{30} a_{32}-a_{31}^2}{a_{30}}$. The $MCNC$ is between the branches of the $LPL$ if and only if $\frac{3 a_{30} a_{32}-a_{31}^2}{a_{30}} > 0$, the result follows in this case.
	
	
	Suppose $\frac{3 a_{30} a_{32}-a_{31}^2}{a_{30}} < 0$. Since $x_K''(0)$ and $x_H''(0)$ have the same sign, the $PC$ and the $MCNC$ are in the same region bounded by the $LPL$. The $MCNC$ is contained between the $PC$ and the $LPL$ when $\vert x_K''(0) \vert > \vert x_H''(0) \vert$, being that
	\begin{equation}\label{eq:sgn}
	\left|x_K''(0)\right| > \left|x_H''(0)\right| \quad \Leftrightarrow \quad (3 a_{30} a_{32}-a_{31}^2) (18 a_{22}^2 a_{30}+\Lambda_{13}) < 0.
	\end{equation}
	Since $\Lambda_{13}$ and $a_{30}$ have the same sign, because $LPL$ has a singularity $A_3^-$, so does $18 a_{22}^2 a_{30}+\Lambda_{13}$. However, $a_{30}$ and $3 a_{30} a_{32}-a_{31}^2$ have opposite signs, so the inequality (\ref{eq:sgn}) is always true.
	
	For a generic surface $(M,p) \in \mathcal{G}_8$, all the conditions above are satisfied. Therefore, the configurations of the curves as in Figure \ref{fig:umb_tipo_luz_desd} middle figures.\vspace{0.2cm}

	\noindent \textbf{(2)} Let $M_t \subset \R^3_1$ be a 1-parameter family of surfaces parameterized by $\bx_t : U \to M_t$, with $\bx_t(x,y) = (x,y,h(x,y,t))$ for some $h$ differentiable with $h(x,y,0) = f(x,y)$. The $LD_t$ is formed by the points $\bx_t(x,y)$ where $\delta_t(x,y) = 0$. We have
	$j^2\delta_t \sim_{\mathcal{R}^{(2)}} \rho(t) + a_{30} x^2 + \frac{\Lambda_{13}}{a_{30}} y^2$
	and $j^1\rho = -2 h_{xt} (0,0,0) t$. Therefore, the $LD_t$ is a $\mathcal{R}-$versal deformation of the $LD$ of $M$ if and only if $h_{xt}(0,0,0) \neq 0$.

	The deformation of the $LPL$ induced by $M_t$ is equivalent to
	\begin{equation}\label{eq:defA3}
	x^2 \pm (y^4 + w_1(t) y^2+w_2(t) y + w_3(t)).
	\end{equation}
	Define the differentiable curve $\gamma : (\R,0) \to \R^3$, given by $\gamma(t) = (w_1(t),w_2(t),w_3(t))$, with $\gamma(0)$ being the origin.
	
	We will consider only the linear part of $\gamma$ to calculate tangent vector to $\gamma$. Therefore, take $h(x,y,t) = f(x,y) +t\tilde{h}(x,y)$ for some smooth function $\tilde{h}$. The $LPL$ of the surfaces $M_t$ is given by $\tilde \delta_t = 0$, with
	\begin{equation}\label{eq:deflpl}
	j^4\tilde \delta_t \sim_{\mathcal{R}} x^2 \pm y^4 - h_{xt}(0,0,0) \sqrt{\frac{32 |\Lambda_{13}|}{3 |a_{30}|}} t y^2.
	\end{equation}
	Thus, the tangent vector to the $\gamma$ at $t=0$ is
	$\left(\pm h_{xt}(0,0,0) \sqrt{\frac{32 |\Lambda_{13}|}{3 |a_{30}|}},0,0\right)$.
	
	Since the $LD$ has a singularity $A_1$, it follows that $p$ is a stable umbilic point \cite{MarcoTari}. Therefore, $\tilde \delta_t$ has a singularity for every $t$, because umbilic points are singularities of the $LPL$, and the curve $\gamma$ is contained in the swallowtail.
	
	As $\gamma$ is regular, it follows from the Theorem \ref{teo:curvregswa} that one part of $\gamma$ lies in the stratum $II$ and the other part in the stratum $VI$, see Figure \ref{fig:desd_A3} left figure. Therefore, the $A_3^\pm$-singularity of the $LPL$ deforms as in Figure \ref{fig:lpl+-}.
	\begin{figure}[h!]
		\centering
		\begin{subfigure}[b]{0.4\textwidth}
			\includegraphics[width=\textwidth]{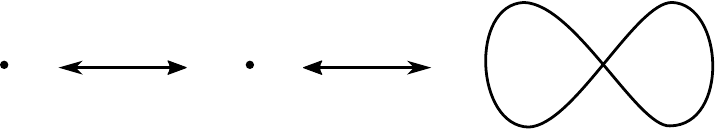}
			\caption{$A_3^+$ Singularity.}
		\end{subfigure}
		\qquad \qquad
		\begin{subfigure}[b]{0.4\textwidth}
			\includegraphics[width=\textwidth]{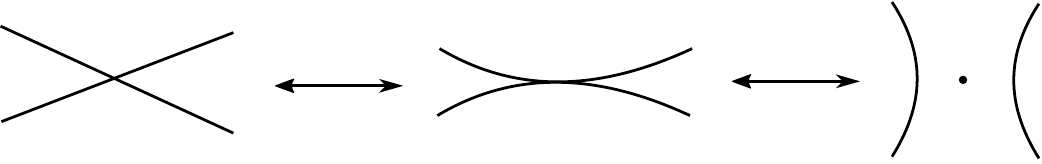}
			\caption{$A_3^-$ Singularity.}
		\end{subfigure}
		\caption{Deformation of the $LPL$ in a generic family of surfaces at a lightlike umbilic point.}
		\label{fig:lpl+-}
	\end{figure}

	Finally, it remains to obtain the intersections between these curves along the $M_t$ family. Let $\theta_{\bx_t}: U \to J^3(2,1)$ be the Monge-Taylor map with respect to $\bx_t$. Put
	$$\theta_{\bx_t} (x,y) = \left( \theta_{00}(x,y,t),\theta_{10}(x,y,t),\theta_{11}(x,y,t),\dots,\theta_{32}(x,y,t),\theta_{33}(x,y,t) \right),$$
	where $\theta_{ij} : U \times (\R,0) \to \R$ is the coefficient of $x^{i-j} y^j$. A point $\bx_t(x,y)$ belongs to the intersection of the $LD$ with the $MCNC$ if and only if $\sigma_t(x,y) = (0,0)$, where $$\sigma_t(x,y) = (\theta_{10}^2+\theta_{11}^2-1,\theta_{10}^2 \theta_{20}+\theta_{10} \theta_{11} \theta_{21}+\theta_{11}^2 \theta_{22})$$
	with $\theta_{ij}$ evaluated in $(x,y,t)$. Note that
	$$\sigma_t \sim_{\mathcal{A}} \left(\frac{\Lambda_{13}}{a_{30}} y^2+2 h_{xt}(0,0,0) t+\rho(t),x\right),$$
	with $j^1\rho = 0$. Therefore, the $LD_t$ and the $MCNC_t$ intersect at 2 or 0 points when $\Lambda_{14} < 0$ or $\Lambda_{14}>0$, respectively, with $\Lambda_{14}=a_{30} h_{xt}(0,0,0) \Lambda_{13} t$.
	The $LD_t$ and the $MCNC_t$ are tangent at $\bx_t(x,y)$ when $T(x,y,t) = 0$, where
	$$\begin{array}{ccl}
	T & = & 2 \theta_{10} \theta_{11}^2 \theta_{20} \theta_{31}-6 \theta_{10} \theta_{11}^2 \theta_{20} \theta_{33} -3 \theta_{10} \theta_{11}^2 \theta_{21} \theta_{30}-\theta_{10} \theta_{11}^2 \theta_{21} \theta_{32}+6 \theta_{11} \theta_{22} \theta_{30}\\ & & -4 \theta_{20} \theta_{21} \theta_{22}+8 \theta_{10} \theta_{11} \theta_{20}^2 \theta_{22} -2 \theta_{10} \theta_{11} \theta_{20} \theta_{21}^2-8 \theta_{10} \theta_{11} \theta_{20} \theta_{22}^2-2 \theta_{10} \theta_{20} \theta_{31}\\ & & +2 \theta_{10} \theta_{11} \theta_{21}^2 \theta_{22}+4 \theta_{11}^3 \theta_{20} \theta_{32}-2 \theta_{11}^2 \theta_{21}^3 -\theta_{11}^3 \theta_{21} \theta_{31}-3 \theta_{11}^3 \theta_{21} \theta_{33}-6 \theta_{11}^3 \theta_{22} \theta_{30} \\ & & +2 \theta_{11}^3 \theta_{22} \theta_{32}+8 \theta_{11}^2 \theta_{20} \theta_{21} \theta_{22}+\theta_{21}^3-4 \theta_{11} \theta_{20} \theta_{32}+\theta_{11} \theta_{21} \theta_{31}+4 \theta_{10} \theta_{11}^2 \theta_{22} \theta_{31} \\ & & +3 \theta_{10} \theta_{21} \theta_{30}.
	\end{array}$$
	If $\eta_t(x,y) = (\sigma_t(x,y),T(x,y,t))$, then $\eta_t \sim_{\mathcal{A}} (2 h_{xt}(0,0,0) t+\rho(t),x,y)$, with $j^1\rho = 0$. Therefore, the $LD_t$ and the $MCNC_t$ are transversal, because $\eta_t(x,y) \neq (0,0,0)$ for $t \neq 0$. 
	
	Similarly, we obtain the intersections between the $LD$ and the $PC$ and between the $MCNC$ and the $PC$. For the intersections between the $LD$ and the $LPL$, use the Theorem \ref{teo:ld_lpl_H}. As $LD_t$ is regular for $t \neq 0$, the intersections between the $LPL_t$ and the $PC$ occur outside the $LD_t$ and the result follows from the Theorem \ref{prop:inter_tripla}.
	
	Given a point $q$ in $LPL_t$, it follows from the Theorems \ref{prop:inter_tripla} and \ref{teo:ld_lpl_H} that $q$ belongs to the $MCNC_t$ if and only if $q$ belongs to the $PC$ or the $LD_t$. Hence, the number of intersection points is $0$, $2$ or $4$. As the $LPL_t$ is tangent to the $LD_t$ and the $PC$, but the $MCNC_t$ is transversal to the $LD_t$ and the $PC$, it follows by transitivity that the $MCNC_t$ is transversal to the $LPL_t$ at the intersection points.

	The result follows from the combination between the deformations of the curves alone and the possible intersections.
\end{proof}

\section{Appendix}\label{sec:appendix}

We consider some regular curves on the discriminant of an $A_3$-singularity. The discriminant of the $\mathcal{R}-$versal unfolding $F(x,y,u,v,w) = x^2 \pm (y^4+u y^2+vy+w)$ of the $A_3^\pm$-singularity $f(x,y) = x^2 \pm y^4$ is given by
$$D_F = \left\{ (u,v,w) \in (\R^3,0) : \exists y \in (\R,0) : y^4+u y^2+v y+w=0, \, 4 y^3+2uy+v=0 \right\}.$$
It is a swallowtail and is parameterized by 
$\varphi(u,y) = (u,-4 y^3 -2uy, 3 y^4+uy^2)$.

Figure \ref{fig:desd_A3} left presents a Whitney stratification of the swallowtail. 
The configuration of the curve $x^2 \pm (y^4+u y^2+vy+w) = 0$ depends only on the stratum in which the parameters $(u,v,w)$ are taken. Figure \ref{fig:desd_A3} presents these configurations.
\begin{figure}[h!]
	\centering
	\begin{subfigure}[b]{0.3\textwidth}
		\includegraphics[width=\textwidth]{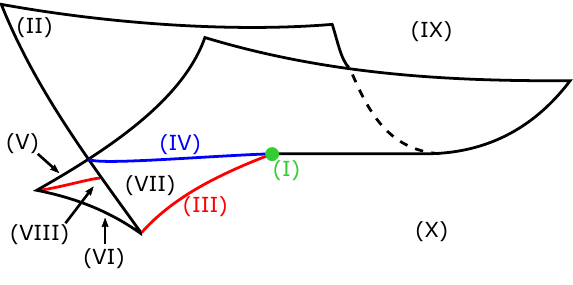}
	\end{subfigure}
	\qquad
	\begin{subfigure}[b]{0.25\textwidth}
		\includegraphics[width=\textwidth]{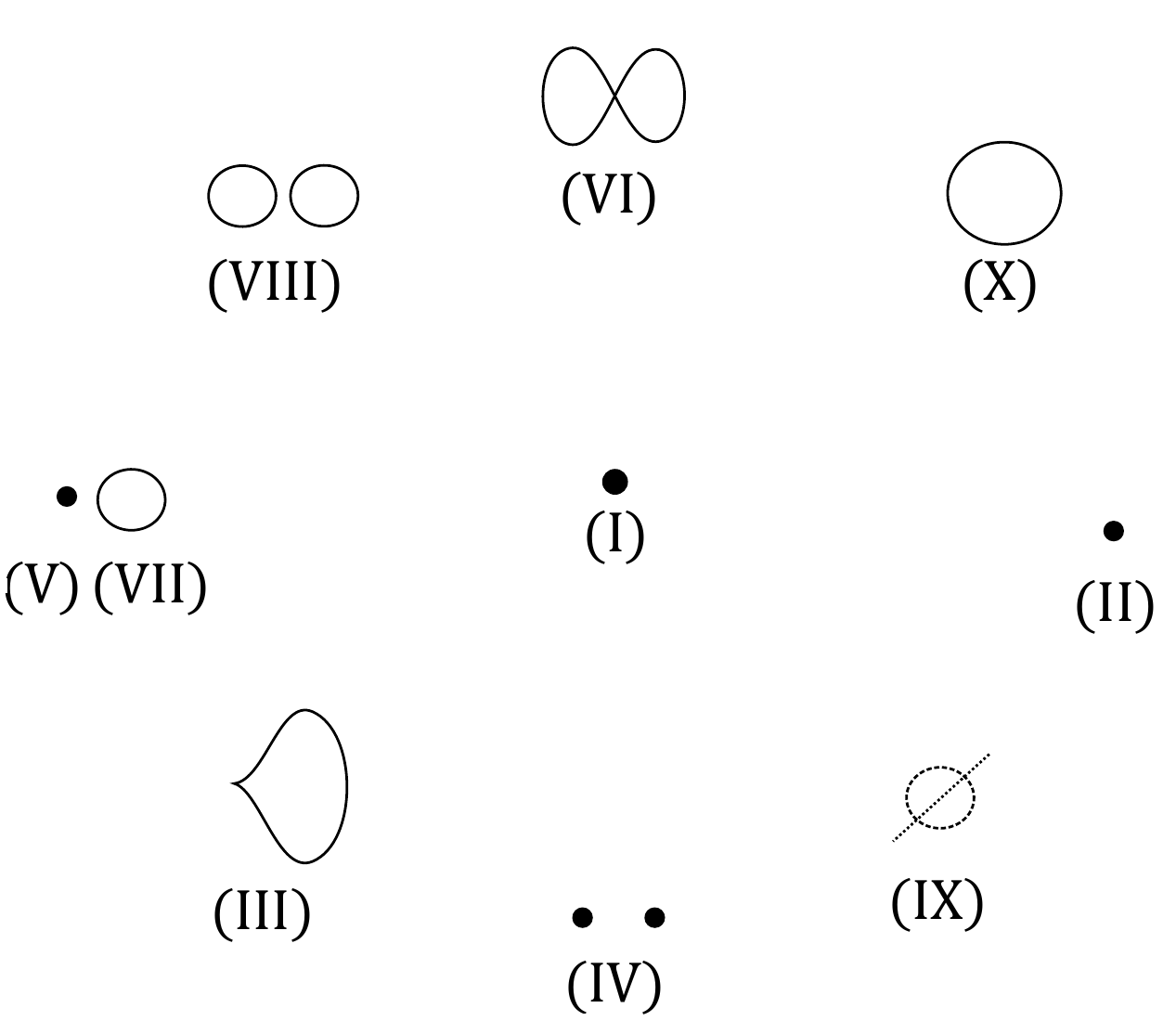}
	\end{subfigure}
	\qquad
	\begin{subfigure}[b]{0.25\textwidth}
		\includegraphics[width=\textwidth]{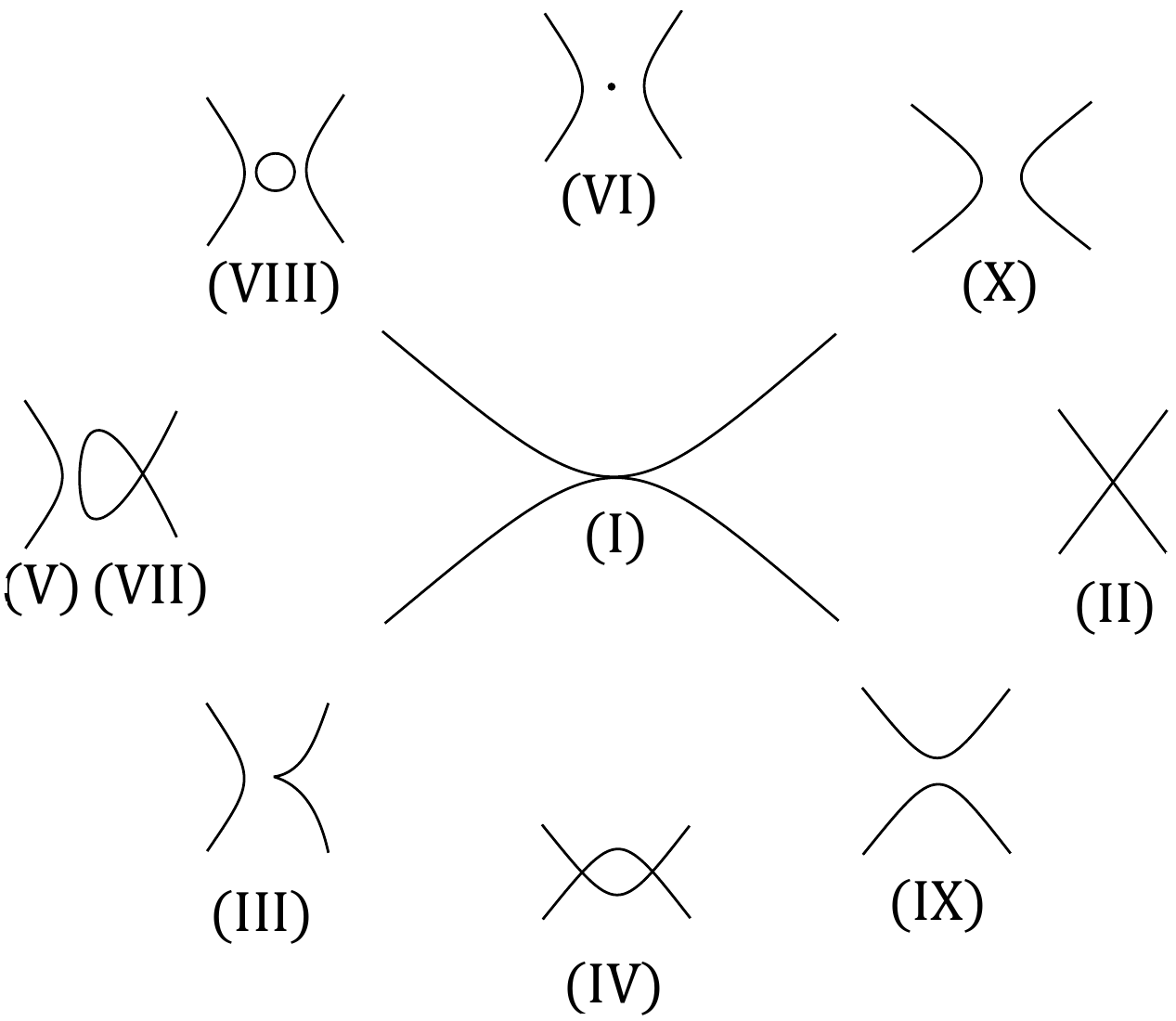}
	\end{subfigure}
	\caption{Whitney stratification of the swallowtail and deformations of $A_3$ singularity ($A_3^+$ middle, $A_3^-$ right) according to the stratum where the parameters (u,v,w) are taken.}
	\label{fig:desd_A3}
\end{figure}

Let $\eta : (-\epsilon,\epsilon) \to \R^3$ be a smooth and regular curve on the swallowtail with $\eta(0)=(0,0,0)$. Then $\eta'(0) = (1,0,0)$ and we can write $\eta (t) = (t,v(t),w(t))$, where $v(t) = -4y(t)^3-2ty(t)$ and $w(t) = 3y(t)^4+ty(t)^2$, for some function $y : (- \epsilon,\epsilon) \to \R$.

\begin{lem}\label{teo:y}
	The function $y(t)$ is smooth at $t=0$.
\end{lem}

\begin{proof}
	Since $\eta$ is smooth, so are $v$ and $w$. Also, $v(t) = t^2 \tilde v(t)$ and $w(t) = t^2 \tilde w(t)$, with $\tilde v$ and $\tilde w$ smooth, because $v(0) = w(0) = v'(0) = w'(0)=0$. Let $P_t(y) = y^4+ty^2+v(t)y+w(t)$ and $Q_t(y) = 4y^3+2ty+v(t)$. Since $\eta$ is a curve on the swallowtail, the resultant 
	$$r(t) = 256 w(t)^3-128 t^2 w(t)^2+16 t^4 w(t)+144t v(t)^2 w(t)-27 v(t)^4-4 t^3 v(t)^2$$
	of the polynomials $P_t$ and $Q_t$ with respect to $y$ is identically zero. It follows that $j^6r = 16 \tilde w(0) (4 \tilde w(0) -1)^2 t^6 \equiv 0$, so $\tilde w(0) = 0$ or $\tilde w(0) = 1/4$. 
	
	Let $Y(t)=y(t)^2$, then $t^2 \tilde w(t)=3Y(t)^2+tY(t)$ and
	$Y(t) = -t \left( \frac{1 \pm \sqrt{1+12 \tilde{w}(t)}}{6} \right)$.
	Since $Y(t)\geq 0$ for all $t \in (-\epsilon,\epsilon)$, it follows that $\tilde w(0)=0$ and $Y(t)=-t \tilde Y(t)$, with
	$\tilde Y(t) = \frac{1 - \sqrt{1+12 \tilde{w}(t)}}{6}$.
	
	On the other hand, $t^2 \tilde v(t) = -4Y(t)y(t)-2ty(t)$. Therefore,
	$y(t)  = \frac{t \tilde v(t)}{4\tilde Y (t)-2}$.
	Since $4\tilde Y(0)-2 = -2 \neq 0$, it follows that $y$ is smooth.
\end{proof}

\begin{theo}\label{teo:curvregswa}
	If $\eta$ is a smooth and regular curve on the swallowtail with $\eta(0) = (0,0,0)$, then one part of $\eta$ lies in the stratum $(II)$ and the other part in the stratum $(VI) $.
\end{theo}

\begin{proof}
	In Figure \ref{fig:desd_A3} left figure, the strata $(III)$ and $(IV)$ are curves parameterized by $\gamma_{1}(t) = (-6t^2,8t^3,-3t^4 )$ and $\gamma_{2}(t) = (-2t^2,0,t^4)$, respectively. The parabolas $(-6t^2,t)$ and $(-2t^2,t)$ in the $uy$ plane are the pre-images of $\gamma_1$ and $\gamma_2$ by the parameterization $\varphi$ of the swallowtail, see Figure \ref{fig:swallowtail_plano}.
	
	\begin{figure}[h!]
		\centering
		\includegraphics[width=0.25\textwidth]{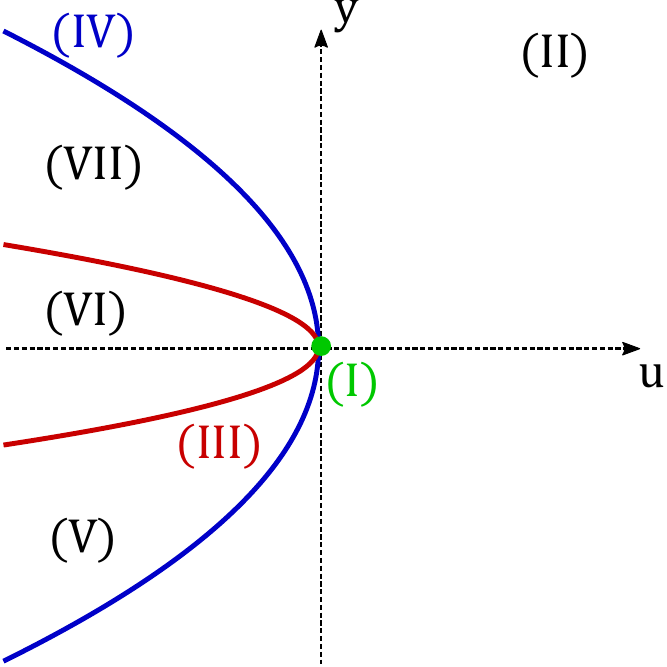}
		\caption{Partition of the $uy$ plane induced by the swallowtail stratification.}
		\label{fig:swallowtail_plano}
	\end{figure}
	
	Since $\eta(t) = (t,v(t),w(t))$ and from the Lemma \ref{teo:y}, it follows that the pre-image of $\eta$ by $\varphi$ is the smooth curve $(t,y(t))$ in the $uy$ plan. Therefore, $\varphi^{-1}(\eta)$ is transversal to the parabolas in Figure \ref{fig:swallowtail_plano}, that is, $\eta$ passes from the stratum $(II)$ to the stratum $(VI) $ when $t$ changes sign.
\end{proof}

\begin{acknow}
This work was financed by the doctoral grant Coordenação de Aperfeiçoamento de Pessoal de Nível Superior – Brasil (CAPES) – Finance Code 001, supervised by Farid Tari.
\end{acknow}


\noindent
MACF: Instituto de Ci\^encias Matem\'aticas e de Computa\c{c}\~ao - USP, Avenida Trabalhador s\~ao-carlense, 400 - Centro, CEP: 13566-590 - S\~ao Carlos - SP, Brazil.\\
E-mail: marcoant\_couto@hotmail.com

\end{document}